\documentclass{article}
\usepackage[english]{babel}
\usepackage{amsmath,amsfonts,amsthm,amssymb,amscd}
 \usepackage{verbatim}

\renewcommand{\Re}{\mathop{\rm Re}\nolimits}
\renewcommand{\Im}{\mathop{\rm Im}\nolimits}

\newcommand{\p}{\partial}
\newcommand{\e}{\varepsilon}

\newcommand{\ri}{{\rightarrow}}
\newcommand{\C}{{\mathbb C}}
\newcommand{\R}{{\mathbb R}}

\newcommand{\N}{{\mathbb N}}
\newcommand{\la}{\lambda}

\newcommand{\ty}{\infty}
\newcommand{\te}{\Theta}
\newcommand{\om}{\omega}

\newcommand{\aA}{{\cal A}}
\newcommand{\BB}{{\cal B}}
\newcommand{\CC}{{\cal C}}
\newcommand{\DD}{{\cal D}}
\newcommand{\EE}{{\cal E}}

\newcommand{\GG}{{\cal G}}
\newcommand{\HH}{{\cal H}}

\newcommand{\UU}{{\cal U}}
\newcommand{\VV}{{\cal V}}

\newcommand{\lag}{\langle}
\newcommand{\rag}{\rangle}

\newcommand{\dd}{{\textup d}}

\theoremstyle{plain}
\newtheorem{theorem}{Theorem}[section]
\newtheorem{lemma}[theorem]{Lemma}%[section]
\newtheorem{proposition}[theorem]{Proposition}
\newtheorem{corollary}[theorem]{Corollary}
\newtheorem{definition}[theorem]{Definition}%[section]
\newtheorem{condition}[theorem]{Condition}
\theoremstyle{remark}
\newtheorem{remark}[theorem]{Remark}

\newcommand{\de}{\delta}
\numberwithin{equation}{section}
\begin{document}

\date{}
\title{Global exact controllability in
infinite time of    Schr\"odinger equation}

\date{}

\author{Vahagn Nersesyan\footnote{Laboratoire de
Math\'ematiques, UMR CNRS 8100, Universit\'e de
Versailles-Saint-Quentin-en-Yvelines, F-78035 Versailles,
France,
e-mail: Vahagn.Nersesyan@math.uvsq.fr}, Hayk
Nersisyan\footnote{Laboratoire de Math\'ematiques, UMR CNRS
8088,
  Universit\'e de Cergy-Pontoise,
 F-95000 Cergy-Pontoise, France, e-mail:
Hayk.Nersisyan@u-cergy.fr}}
 \maketitle

{\small\textbf{Abstract.} In this paper, we study the problem of controllability of Schr\"odinger equation. We prove that the system is   exactly controllable  in infinite time to  any position. The proof is based on an inverse mapping theorem for multivalued functions.   We show also that the system   is not exactly controllable  in finite time in lower Sobolev spaces.
  }\\

  {\small\textbf{R\'esum\'e.} Dans cet article, nous \'etudions le probl\`eme de contr\^olabilit\'e pour l'\'equation de Schr\"odinger. Nous montrons que le syst\`eme est   exactement contr\^olable en temps infini. La preuve est bas\'ee sur un th\'eor\`eme d'inversion locale pour des multifonctions.  Nous montrons aussi que le syst\`eme n'est pas exactement contr\^olable en temps fini dans les espaces de Sobolev d'ordre inf\'erieur.   }
  \\

  Keywords: Schr\"odinger equation, controllability, multivalued mapping, Kolmogorov
$\e$-entropy
  \tableofcontents

\section{Introduction}\label{S:intr}
The paper is devoted to the study of the following controlled Schr\"odinger equation
\begin{align}
i\dot z &= -\Delta z+V(x)z+ u(t) Q(x)z,\,\,\,\, \label{E:hav1}\\
z\arrowvert_{\partial D}&=0,\label{E:ep1}\\
 z(0,x)&=z_0(x).\label{E:sp1}
\end{align}  We assume that   space variable $x$ belongs to a rectangle  $D\subset\R^d$, $V,Q \in C^\ty(\overline{D},\R)$ are   given functions,  $u$ is the control,
and $z$ is the state. 
  We prove that the linearization of this system is
  exactly controllable in Sobolev spaces in infinite
time. Application of this result gives global exact controllability in infinite time in $H^3$ for $d=1$.
  We show also that   the system is not exactly
controllable in finite time in lower Sobolev spaces.

Let us recall some previous  results on the controllability problem of Schr\"odin\-ger equation.
  In \cite{BCH}, Beauchard  proves an
exact controllability result for the system with  $d=1, D=(-1,1 )$ and $Q(x)=x$
   in $H^7$-neighborhoods of the  eigenfunctions.  Beauchard and Coron
\cite{BeCo} established later a partial global exact
controllability result, showing that the system in question
  is also controlled between  
neighborhoods of   eigenfunctions. Recently,  Beauchard and  Laurent \cite{KBCL} simplified the proof of \cite{BCH} and generalized it to the case of the nonlinear equation.
The proofs of \cite{BCH, BeCo, KBCL} work also for the neighborhoods of finite linear combinations of eigenfunctions.  In the case of infinite linear combinations, these arguments do not work, since the linearized system does not verify the property of spectral gap (even if the problem is 1-D), hence the Ingham inequality cannot be applied.

Chambrion et al.
\cite{CH}, Privat, Sigalotti \cite{PISI}, and Mason,   Sigalotti  \cite{PMMS} prove that (\ref{E:hav1}),
(\ref{E:ep1})  is  approximately controllable in~$L^2$ generically with respect to   function    $Q$ and   domain $D$. In \cite{VaNe, VN},  the first author of this paper proves  a stabilization result and a property of global approximate controllability  to eigenstates   for Schr\"odinger equation. Combination  of these  results  
with the local exact controllability property obtained by Beauchard \cite{BCH} gives   global exact controllability   in finite time for $d=1$ in the spaces $H^{3+\e}, \e>0$.
  See also  
papers  \cite{RSDR,TR,ALT,ALAL,AC,BCMR}     for controllability of finite-dimensional
systems and   papers
\cite{L,MZ,BP,Z,DGL,Mir,PERV} for controllability properties of   various  Schr\"odinger systems.

In this article, we   study the properties of control system on the   time half-line $\R_+$ instead of a finite interval $[0,T]$, as in all above cited papers. We study the mapping, which  associates to initial condition $z_0$ and  control $u$ the $\omega$-limit set of the corresponding trajectory. We consider this mapping as a multivalued function in the phase space. We show that this multivalued function is differentiable with differential equal to the limit of the linearization of (\ref{E:hav1}), (\ref{E:ep1}), when time $t$ goes to infinity.  Observing  that the linearized system is controllable in infinite time at almost any point,  we conclude the controllability of the nonlinear system (in the case $d=1$), using an inverse mapping theorem for multivalued functions \cite{KNJP} by Nachi and  Penot.    Thus   (\ref{E:hav1}), (\ref{E:ep1}) is exactly controllable near   any point in the phase space, hence globally. The controllability of   the linearized system is proved for any $d\ge1$, but this result is not directly applicable  to the study of the nonlinear system with $d\ge 2$.  We have a loss of regularity: the solution of the nonlinear problem exists for more regular controls than the ones   used to control the linear problem. The multidimensional case is treated in our forthcoming paper.

To our knowledge, the inverse mapping theorem for multivalued functions was never used before in the theory of control of PDEs. Our proof does not rely on the particular asymptotics of the eigenvalues of Dirichlet Laplacian, so it is likely to  work in     other    settings. Considering the problem in infinite time enables us   to prove the controllability of the linearized system in the case of any space dimension $d\ge1$, even when the gap condition is not verified for the eigenvalues (which is the case for $d\ge 3$).

In the second part of the paper, we study the problem of non-controllability for  (\ref{E:hav1}), (\ref{E:ep1}) in finite time. We prove that the system is not exactly controllable in finite time in the spaces $H^k$ with $k\in (0,d)$. Let us recall that previously 
  Ball,  Marsden and Slemrod
\cite{BMS} and Turinici \cite{GBT} have shown that  the problem is not controllable in the space    $  H^2$. Our result is inspired by the paper \cite{SHEU}  of Shirikyan, where the non-controllability of 2D Euler equation is established. More precisely, it is proved in \cite{SHEU}  that, if  the Euler system is controlled by finite dimensional external force, then   the set of all reachable points in a given time $T>0$ cannot cover a ball in the phase space. Later this result was generalized by the second author of the present paper, in \cite{HANE}: in the case of 3D Euler equation it is proved that the union of all sets of reachable points at all times $T>0$ also does not cover a ball.

Using ideas  of Shirikyan, we prove  that the image by the resolving operator   of a ball in the space of   controls  has a Kolmogorov $\e$-entropy strictly less than that  
of a ball   in the phase space $H^k$. This implies the non-controllability.

\vspace{10pt}

\textbf{Acknowledgments.} The authors would like to thank Armen Shirikyan for many fruitful conversations.

\vspace{22pt} \textbf{Notation}\\\\ In this paper, we use the
following notation. 
Let \begin{align*}
\ell^2&:=\{\{a_j\}\in \C^\ty: \|\{a_j\}\|_{\ell^2}^2=\sum_{j=1}^{+\ty} |a_j|^2<+\ty\}\\
\ell^2_0&:=\{\{a_j\}\in \ell^2: a_1\in \R\}.
\end{align*}  We denote by  $H^s:=H^s(D)$    the Sobolev space of order $s\ge0$.
Consider the Schr\"odinger operator  $ -\Delta +V $, $V \in C^\ty(\overline{D},\R)$  with $ \DD(- \Delta +V):=H_0^1
\cap H^2 $. Let
$\{\la_{j,V} \}$ and $\{e_{j,V} \}$ be the sets of eigenvalues and
normalized eigenfunctions  of this operator.  Let $\langle\cdot,\cdot\rangle$ and~$\|\cdot\|$ be the scalar
product and the norm in the space $L^2 $. Define the space~$H^s_{(V)}:=D((-\Delta+V)^\frac{s}{2})$ endowed with the norm $~\|\cdot\|_{s,V}= \|(\la_{j,V})^\frac{s}{2} \langle\cdot,e_{j,V}\rangle\|_{\ell^2}$.
 When $D$ is the rectangle  $(0,1)^d $      and $V  (x_1,\ldots,x_d) =V_{1} (x_1)+\ldots+ V_{d} (x_d)  $, $V_k   \in C^\ty( [0,1],\R)$,   the eigenvalues and eigenfunctions of $ -\Delta +V $ on $D$ are of the form  
\begin{align}
\la_{j_1,\ldots, j_d,V}&=\la_{j_1,V_1}+\ldots+\la_{j_d,V_d},\label{E:j1}\\
e_{j_1,\ldots, j_d,V}(x_1,\ldots,x_d)&=e_{j_1,V_1}(x_1)\cdot\ldots\cdot e_{j_d,V_d}(x_d),\,\,\,\,\, (x_1,\ldots,x_d)\in D, \label{E:j2}
\end{align}where $\{\la_{j,V_k} \}$ and $\{e_{j,V_k} \}$ are the     eigenvalues and
  eigenfunctions  of operator $-\frac{\dd^2}{\dd x^2}+V_k$  on $(0,1)$.    
 Define the space
 \begin{equation}\label{E:ddsah}
 \VV:=\{z\in L^2:\| z \|_{\VV}^2:=\!\!\!\!\!\!\sum_{j_1,\ldots,j_d=1}^{+\ty}\!\!\!| (  {{j_1^3\cdot\ldots\cdot j_d^3}} \lag z, e_{{j_1,\ldots, j_d,V}}   \rag|^2<+\ty\} .\end{equation}    Notice that, in the case $d=1$, the space $\VV$ coincides with $H^3_{(V)}$. The eigenvalues and eigenfunctions of Dirichlet Laplacian on the interval  $(0,1)$ are   $\la_{k,0}=   k^2\pi   ^2$ and $e_{k,0} (x)=\sqrt {2}  \sin  ( {k\pi} x  ) $, $x\in(0,1)$.  It is well known that for any $V\in L^2([0,1],\R)$
 \begin{align}
& \la_{k,V}=k^2\pi^2+\int_0^1 V(x)\dd x+r_k,\label{E:app1}\\
 & \|e_{k,V}-e_{k,0}\|_{L^\ty}\le \frac{C}{k},\label{E:app2}\\
 &  \Big\|\frac{\dd e_{k,V}}{\dd x}-\frac{\dd e_{k,0}}{\dd x} \Big\|_{L^\ty}\le C,\label{E:app3}
 \end{align} where $\sum_{k=1}^{+\ty}r_k^2<+\ty$ (e.g., see \cite{PT}).
   For a Banach space $X$, we shall denote by $B_X(a, r)$ the open ball of radius $r > 0$ centered at  $a\in X$. For a set $A$, we write $2^A$ for the set consisting of all subsets of $ A$. We denote by $C$ a constant whose value may change from line to line.

 \section{Controllability of   linearized
system}\label{S:MR}

\subsection{Main result}\label{S:Gcayin}
In this section, we suppose that $d=1$ and $D=(0,1)$. 
For any $\tilde z \in H^3_{(V)}$, let $\UU_t(\tilde z ,0) $ be the solution of (\ref{E:hav1})-(\ref{E:sp1}) with $z_0=\tilde z$ and $u=0$. Clearly,
\begin{equation}\label{E:azh}\UU_t(\tilde z ,0) = \sum_{j=1}^{+\ty}e^{-i\la_{j,V}t} \lag \tilde z,e_{j,V}\rag e_{j,V}.\end{equation}

\begin{lemma}\label{L:mlk}
There is a sequence $T_n\ri+\ty$ such that for any $ \tilde z\in H^3_{(V)} $ we have  $  \UU_{T_n}(\tilde z ,0)\ri \tilde z $ in $ H^3_{(V)}$.
\end{lemma}
\begin{proof}The proof uses the following well known result (e.g., see \cite{KAHA}).
\begin{lemma}\label{L:AQ}
For any $\e>0$, $N\ge 1$ and   $\alpha_j\in\R$, $j=1,\ldots,N$, there is  $k\in \N$ such that 
$$\sum_{j=1}^N  |e^{i\alpha_jk}-1 |<\e.
$$
\end{lemma} Applying this lemma, we see that for any $\e>0$ and for sufficiently large  $N\ge 1$, we have
\begin{align}
\| \UU_k(\tilde z ,0)-\tilde z\|_{3,V}^2\le &\sum_{j\le N}   |e^{-i\la_{j,V}k}-1 |^2|\la_{j,V}^\frac{3}{2}\lag \tilde z,e_{j_1,\ldots,j_d,V}\rag|^2\nonumber\\&+2 \sum_{j> N}    |\la_{j,V}^\frac{3}{2}\lag \tilde z,e_{j_1,\ldots,j_d,V}\rag|^2  \le \frac{\e}{2}+ \frac{\e}{2} =\e \nonumber
\end{align} for an appropriate choice of $k\in \N$. This proves Lemma \ref{L:mlk}.

\end{proof}
This subsection is devoted to the  study of the linearization of   (\ref{E:hav1}), (\ref{E:ep1}) around the trajectory $\UU_t(\tilde z ,0) $:
\begin{align}
i\dot z &= -\frac{\p^2z}{\p x^2} +V(x)z  + u(t) Q(x)\UU_t(\tilde z,0),\,\,\,\,
\label{E:hav2}\\
z\arrowvert_{\partial D}&=0,\label{E:ep2}\\
 z(0,x)&=z_0.\label{E:sp2}
\end{align}      Let $S$ be the unit sphere in $L^2$. For  $y\in S$, let $T_y$ be the   tangent space to $S$ at $y\in S$:
$$
T_{y}=\{z\in S: \Re\lag z, y\rag=0\}.
$$
\begin{lemma}\label{L:Linlav}
 For any $z_0\in T_{\tilde z}\cap H_{(0)}^2$ and $u\in L_{loc}^1(\R_+,\R)$,
  problem
(\ref{E:hav2})-(\ref{E:sp2}) has a unique solution   $ z\in
C(\R_+, H_{(0)}^2) $.
 Furthermore, if $R_t(\cdot, \cdot):
T_{\tilde z}\cap  H_{(0)}^2\times L^1([0,t],\R)\rightarrow
 H_{(0)}^2 $, $(z_0,u)\rightarrow z(t)$ is the resolving operator of the problem,
then
\begin{enumerate}
 \item[(i)] $R_t(z_0,u)\in T_{\UU_t(\tilde z,0)}$ for any
$t\ge0$,
 \item[(ii)] The operator $R_t$ is linear  continuous from
 $T_{\tilde z}\cap  H_{(0)}^2\times L^1([0,t],\R)$ to $ H_{(0)}^2 $.
       \end{enumerate}
  \end{lemma}\begin{proof} 
  The proof of existence and (ii)  is standard (e.g., see \cite{CW}).    
   To prove (i), notice that
  \begin{align*}
  \frac{\dd}{\dd t}\Re \lag R_t , \UU_t \rag&= \Re \lag \dot R_t , \UU_t \rag+ \Re \lag R_t , \dot \UU_t \rag\\&= \Re \lag i(\frac{\p^2}{\p x^2}-V) R_t  -i u(t) Q(x)\UU_t 
 , \UU_t \rag + \Re \lag R_t , i(\frac{\p^2}{\p x^2}-V) \UU_t \rag \nonumber\\
 &=\Re \lag i(\frac{\p^2}{\p x^2}-V) R_t   
 , \UU_t \rag + \Re \lag R_t , i(\frac{\p^2}{\p x^2}-V) \UU_t \rag=0.
 \end{align*} Since $ \Re \lag R_0 , \UU_0 \rag=  \Re \lag z_0 , \tilde z \rag=0$, we get (i).
  \end{proof}
As (\ref{E:hav2})-(\ref{E:sp2}) is a linear control problem, the controllability of system with $z_0=0$ is equivalent to that   with any $z_0\in T_{\tilde z}$. Henceforth, we   take $z_0=0$ in   (\ref{E:sp2}).
Let us rewrite this problem in the Duhamel form
\begin{equation}\label{E:lpo}z(t)= -i\int_0^t S(t-s) u(s) Q(x)\UU_s(\tilde z,0)\dd s,
\end{equation} where $S(t)=e^{it(\frac{\p^2}{\p x^2}-V)}$ is the free evolution.
Using (\ref{E:azh}) and (\ref{E:lpo}), we obtain
\begin{equation}\label{E:gcz}
\lag z(t), e_{m,V}\rag=-i \sum_ {k=1} ^{+\ty} e^{-i\la_{m,V}t}\lag\tilde z,e_{k,V}\rag Q_{mk} 
\int^t_0 e^{i\om_{mk} s} u(s)\dd s,\,m\ge1,
\end{equation}where $\om _{mk}=\la_m-\la_k$ and  $Q_{mk}:= \langle
Qe_{m,V },e_{k,V }\rangle $. Let $T_n\rightarrow+\ty$ be the sequence in Lemma \ref{L:mlk}.
Then  
$e^{-i\la_{m,V}T_n}\rightarrow 1$ as $n \rightarrow +\ty$. Let us take  $t=T_n$
in (\ref{E:gcz}) and pass to the limit as $n \rightarrow +\ty$. For any
$u\in L^1(\R_+,\R)$  the right-hand side has a limit.   Equality   (\ref{E:gcz}) 
 implies that  
  the following limit exists    in the $L^2$-weak sense 
\begin{equation}\label{E:eez}R_{\ty}(0,u):=\lim_{n\rightarrow +\ty} z(T_n)=\lim_{n\rightarrow
+\ty}
R_{T_n}(0,u).
\end{equation}  The choice of the sequence $T_n$ implies that
\begin{equation}\label{E:gcza}
\lag R_{\ty}(0,u), e_{m,V}\rag=-i \sum_ {k=1} ^{+\ty} \lag\tilde z,e_{k,V}\rag Q_{mk} 
\int^{+\ty}_0 e^{i\om_{mk} s} u(s)\dd s.
\end{equation}
 Moreover, $R_{\ty}(0,u) \in T_{\tilde z} $. Indeed, using (\ref{E:eez}) and  the convergence $  \UU_{T_n}(\tilde z ,0)\ri \tilde z $ in $ H^3_{(V)}$, we get 

$$
\Re\lag R_{\ty}(0,u),\tilde z\rag=\lim_{n\ri\ty} \Re\lag R_{T_n}(0,u),\UU_{T_n}(\tilde z,0)\rag=0,
$$ by property (i).

For any $u\in L^1(\R_+,\R)$,  denote by $\check u$ the inverse Fourier transform of the function obtained by extending $u$ as zero to $\R_-^*$:
\begin{equation}\label{E:eaz}
\check u(\om):= \int^{+\ty}_0 e^{i\om s} u(s)\dd s.
\end{equation} Define the following spaces   
\begin{align*}
 \tilde\ell^2&:=\{d=\{d_{mk}\} :    \|d\|_{\tilde\ell^2}^2 :=|d_{11}|^2+\sum_{m,k=1, m\neq k}^{+\ty}  |d_{mk}|^2  <+\ty, d_{mm}=d_{11} \\& \quad \quad \quad     \quad \quad \quad \quad \quad \quad \quad \quad \quad \quad\text{and $d_{mk}= {\overline d}_{km}$ for all $m,k\ge1$}   \},
 \\\BB&:=\{u\in L^2_{loc}(\R_+, \R): \|u\|_ \BB ^2:= \sum_{p=1}^{+\ty} p^2 \|u\|_{L^2([p-1,p])}^2<+\ty\},\\  \CC&: =\{u\in L^1(\R_+,\R): \{\check u (\om_{mk}) \}\in \tilde\ell^2\}  .
\end{align*} The set of admissible controls is the Banach space
$$
\Theta:= u\in  \BB\cap   \CC\cap H^s(\R_+,\R) $$
endowed with the norm
$\|u\|_\te:= \|u\|_{\BB}+\|u\|_{L^1}+\|\{\check u (\om_{mk})\}\|_{  \tilde\ell^2}+\|u\|_{H^s}$,
where $s\ge1$ is any fixed constant. Clearly, the space $  \te$ is nontrivial.
 The presence of the space $\BB$ in the definition of $\te$ is motivated by the application to the nonlinear control system that we give in Section \ref{S:NL} (this guarantees   that the trajectories of the nonlinear system with controls from $\BB$ are bounded in the phase space). The space $\CC$ in the definition of  $\Theta$ ensures that the operator $R_{\ty}(0,\cdot)$ takes its values in $H^3_{(V)}$.   

\begin{lemma}\label{L:gc1}
For any $\tilde z\in S \cap H^3_{(V)}  $, $R_{\ty}(0,\cdot)$ is linear continuous mapping from $\te$ to $
T_{\tilde z}\cap H^3_{(V)}$.
\end{lemma}
\begin{proof}  \vspace{6pt}\textbf{Step 1.} Let us admit that for any $m,k\ge 1$  we have
\begin{align}\label{E:rezag1}
\Big|\frac{m ^3}{ k ^3} \langle Qe_{k ,V },e_{m ,V }\rangle\Big|\le C.
\end{align}
Then (\ref{E:app1}), (\ref{E:gcza}), (\ref{E:rezag1}) and the Schwarz inequality imply that
\begin{align*} 
\|   R_{\ty}(0,u)\|_{3,V}^2&\le C \sum_{m=1 }^{+\ty} |  {{m ^3 }} \lag R_{\ty}(0,u), e_{{m ,V}}   \rag|^2\\&\le C\!\!\!\sum_{m=1}^{+\ty}\!\! \Big|\! {{m^3}} \lag \tilde z, e_{{m,V}}   \rag\langle Qe_{m,V },e_{m,V }\!\rangle \!\!\int^{+\ty}_0\!\!\!\!  u(s)\dd s\Big|^2\!\\&\quad + C \!\| \tilde z\|_{3,V} ^2 \!\!\!\sum_{m,k=1, m\neq k }^{+\ty}\!\! \Big|\!\frac{m^3}{k^3}\! \langle Qe_{k,V },e_{m,V }\!\rangle \!\!\int^{+\ty}_0\!\!\!\! e^{i\om_{mk} s} u(s)\dd s\Big|^2\\&\le C\| \tilde z\|_{3,V} ^2 \|u\|_\te^2<+\ty.
\end{align*} 

 \vspace{6pt}\textbf{Step 2.} Let us prove (\ref{E:rezag1}). 
 Integration by parts gives
\begin{align*}
\lag Q e_{k,V} ,  e_{m,V} \rag=&\frac{1}{\la_{m,V}^2} \lag   (-\frac{\p^2}{\p x^2} +V) (Q e_{k,V}),  (-\frac{\p^2}{\p x^2} +V) (e_{m,V} ) \rag \nonumber\\=& \frac{1}{\la_{m,V}^2} (-2\frac{\p Q}{\p x}  \frac{\p e_{k,V}}{\p x}    \frac{\p e_{m,V} }{\p x} \Big|_{x=0}^{x=1} \nonumber\\&+ \lag   \frac{\p}{\p x}   (-\frac{\p^2}{\p x^2} +V) (Q e_{k,V}),   \frac{\p e_{m,V}}{\p x}  \rag\nonumber\\&+ \lag   (-\frac{\p^2}{\p x^2} +V) (Q e_{k,V}),   Ve_{m,V}  \rag).
\end{align*} 
 
 In view of (\ref{E:j1})-(\ref{E:app3}), this implies   (\ref{E:rezag1}).

\end{proof}

We prove the controllability of (\ref{E:hav2}),  (\ref{E:ep2}) under   below
condition with $d=1$. 

\begin{condition}\label{C:p1}Suppose that $D$ is the rectangle  $(0,1)^d $, $d\ge 1$    and  
the   functions  $V,Q \in C^\ty(\overline{D},\R)$ are such that
\begin{enumerate}
\item [(i)] $\inf_{p_1, j_1,\ldots,p_d,j_d\ge1}|{{(p_1j_1 \cdot\ldots\cdot p_dj_d)^3}}   Q_{pj}|\!>\!0,   \! Q_{pj}\!:=\!  \langle Qe_{p_1,\ldots,p_d,V },e_{j_1,\ldots,j_d,V }\rangle $,
   \item [(ii)] $\la_{i,V }-\la_{j,V }\neq \la_{p,V }-\la_{q ,V}$ for
all $i,j,p,q\ge1$ such that
  $\{i,j\}\neq\{p,q\}$ and $i\neq j$.

\end{enumerate}
\end{condition}
See Appendix for the proof of genericity of this condition. 
Let  us introduce the set   \begin{align*}
\EE\!:=\!\{z\in S\!:\!\exists p,q\ge 1,p\neq q,& z=c_pe_{p,V}+c_qe_{q,V}, \nonumber\\& |c_p |^2 \lag Q e_{p,V},e_{p,V}\rag\!-\!| c_q |^2 \lag Q e_{q,V},e_{q,V}\rag = 0  \}.\end{align*}  The following result is proved in next subsection.
\begin{theorem}\label{T:Lin} Under Condition  \ref{C:p1} with $d=1$, for any $\tilde z\in S \cap H^3_{(V)} \setminus \EE$, the mapping
$R_{\ty}(0,\cdot):\te\rightarrow
T_{\tilde z}\cap H^3_{(V)}$   admits a continuous
right inverse, where the space $T_{\tilde z}\cap H^3_{(V)}$ is endowed with the
norm of $H^3_{(V)}$. If $\tilde z\in S \cap H^3_{(V)}\cap \EE$, then $R_{\ty}(0,\cdot)$ is not invertible.
\end{theorem}
\begin{remark}
The invertibility of the mapping $R_{T}(0,\cdot)$  with finite $T>0$ and $\tilde z=e_1$ is studied by Beauchard et al. \cite{BCKY}. They prove that for   space dimension $d\ge 3$ the  mapping  is not invertible.  By Beauchard \cite{BCH}, $R_{T}$ is invertible in the case $d=1$ and $\tilde z=e_1$. The case $d=2$ is open to our knowledge.\end{remark}
\begin{remark} Let us emphasize that the set $\{\om_{mk}\}$ does not verify the gap condition (even in the case $d=1$)
$$
\inf _{(m,k)\neq(m',k')}|\om_{mk}-\om_{m'k'}|>0.
$$ Thus one cannot prove exact controllability in finite time near points, which are not eigenfunctions, using arguments based on the Ingham inequality.

\end{remark}

\subsection{Proof of Theorem \ref{T:Lin}}

The proof of the theorem is based on the following proposition, which is proved in next subsection.
\begin{proposition}\label{P:P1} If the sequence $\omega_m\in \R, m\ge1$ is such that $\om_1=0$ and  $\sum_{m=2}^\ty\frac{1}{|\omega_m|^p}<+\ty$  for some $p\ge1$ and $\om_i\neq \om_j$ for $i\neq j$, then
there is a linear continuous operator $A$ from $ \ell^2_0$ to $  \te$ such that $\{ \check { A(d)} (\om_{m})\}=d$ for any $d\in  \ell^2_0 $. 
\end{proposition}The idea of the proof of Theorem \ref{T:Lin}  is to rewrite (\ref{E:gcza}) in the form $d_{mk} = \check u (\om_{mk}) $ with $d\:=\{d_{mk}\}\in \tilde\ell^2 $ and to apply the proposition. Notice that $\sum_{m,k=1, m\neq k}^\ty \frac{1}{\om_{mk}^4}<+\ty$ and  $\om_{ij}\neq \om_{pq}$ for
all $i,j,p,q\ge1$ such that
  $\{i,j\}\neq\{p,q\}$ and $i\neq j$.
  Let us take any $y\in  T_{\tilde z}\cap H^3_{(V)}$. Define \begin{equation*} d_{mk}:=\frac{i \lag y, e_{m}\rag \lag e_{k},\tilde z\rag -i\lag  e_{k},y\rag \lag \tilde z,e_{m}\rag}{ Q_{mk}  
}+C_{mk},
\end{equation*} where $C_{mk}\in\C$ and $e_k={e_{k,V}}$. The   fact that $\tilde z\in S$ implies  \begin{align*}
 -i \sum_ {k=1} ^{+\ty} \lag\tilde z,e_{k}\rag Q_{mk}  
d_{mk}&= \sum_ {k=1} ^{+\ty} \lag y, e_{m}\rag | \lag \tilde z, e_{k}\rag |^2  -\sum_ {k=1} ^{+\ty} \lag  e_{k},y\rag \lag \tilde z,e_{m}\rag \lag\tilde z,e_{k}\rag
\\&\quad-i \sum_ {k=1} ^{+\ty} \lag\tilde z,e_{k}\rag Q_{mk}  C_{mk}\\&= \lag y, e_{m}\rag-\lag \tilde z, e_{m}\rag \lag \tilde z, y\rag  -i \sum_ {k=1} ^{+\ty} \lag\tilde z,e_{k}\rag Q_{mk}  C_{mk}.
 \end{align*} By (\ref{E:gcza}), we have  $y=R_{\ty}(0,u)$, when 
 \begin{align} &  i  \sum_ {k=1} ^{+\ty}  \lag\tilde z,e_{k}\rag Q_{mk}  C_{mk}=-\lag \tilde z, e_{m}\rag \lag \tilde z, y\rag \label{E:C1}\end{align}
 for all $m\ge1$. Thus if we show that 
  there are $C_{mk}\in\C$ such that (\ref{E:C1}) is verified   and  $d\:=\{d_{mk}\}\in \tilde\ell^2 $, then 
   the proof of the theorem will be completed, in view of Proposition \ref{P:P1}. 
 Notice that, under Condition \ref{C:p1}, we have
 \begin{align*}
\sum_{m,k=1,m\neq k}^{+\ty}\Big| \frac{  \lag y, e_{m}\rag \lag e_{k},\tilde z\rag  }{ Q_{mk}  
} \Big |^2 &\le   C \| y\|_{3,V}^2\|\tilde z\|_{3,V}^2<+\ty.
\end{align*}
Thus  $ \{d_{mk}\}\in \tilde\ell^2 $, if   $C_{mk}\in\C$ are such that  \begin{align}  &d_{mm} =\frac{i \lag y, e_{m}\rag \lag e_{m},\tilde z\rag  -i\lag  e_{m},y\rag \lag \tilde z,e_{m}\rag}{ Q_{mm}  
}+C_{mm} =d_{0}, \label{E:C2}\\ &C_{mk} =\overline C_{km},\label{E:C3}\\ &\sum_{m,k=1, m\neq k}^{+\ty}   |C_{mk}|^2  <+\ty,\label{E:C4}
 \end{align} where $d_{0}\in\R$. Let us show that, for an appropriate choice of $d_0$, there are $C_{mk}$ satisfying (\ref{E:C1})-(\ref{E:C4}). Since $ y \in T_{\tilde z}$, we have $\lag \tilde z, y\rag =i \Im \lag \tilde z, y\rag $. We can rewrite (\ref{E:C1}) and (\ref{E:C2}) in the following form
 \begin{align} &     \sum_ {k=1} ^{+\ty}  \lag\tilde z,e_{k}\rag Q_{mk}  C_{mk}=-\lag \tilde z, e_{m}\rag \Im \lag \tilde z, y\rag,\label{E:C21}\\&d_{mm} =\frac{-2\Im( \lag y, e_{m}\rag \lag e_{m},\tilde z\rag)}{ Q_{mm}  
}+C_{mm} =d_{0}. \label{E:C22} 
 \end{align}
 
 \vspace{6pt}\textbf{Case 1.} Let as suppose that $\tilde z= ce_p$, where $c\in\C, |c|=1$ and $p\ge1$. Then (\ref{E:C3})-(\ref{E:C22}) is verified for $C_{mk}=0$,   if $m\neq k$ and $C_{mm} $ defined by  (\ref{E:C22}) with  $d_0=\frac{\Im \lag \tilde z,y\rag}{Q_{pp}}$.
 
 \vspace{6pt}\textbf{Case 2.} Suppose  $\tilde z= c_pe_p+c_qe_q$, where $c_p,c_q\in\C,  {|c_p|^2+|c_q|^2}=1$ and $p\neq q$. For any  $m\ge 1$, define $C_{mm} $   by  (\ref{E:C22}). If $m\neq p$, we  set  
 \begin{equation}\label{E:Cmp}
 C_{mp}:=\frac{-c_m(\Im \lag \tilde z, y\rag+Q_{mm}C_{mm})}{c_pQ_{mp}},
\end{equation}where $c_m=0$ for $m\neq q$, and $C_{mk}=0$ for any $k\ge 1$ such that $k\neq m, p$. Then all the equations in (\ref{E:C21}) are verified, excepted the case   $m=p$. 
Let us show that, for an appropriate choice of   $d_0\in \R$, this equation is also satisfied. Equation  (\ref{E:C21}) for   $m=p$ is
 $$
 c_pQ_{pp}C_{pp}+ c_qQ_{pq}C_{pq}=-c_p \Im \lag \tilde z, y\rag.
 $$Using (\ref{E:Cmp}) for $m=q$ (taking $C_{pq}=\overline{C}_{qp}$) and (\ref{E:C22}) for $m=p$, we get
\begin{align*}
-c_p \Im \lag \tilde z, y\rag=& c_pQ_{pp}\Big(d_0+\frac{2\Im( \lag y, e_{p}\rag \lag e_{p},\tilde z\rag)}{ Q_{pp}  
}\Big)\\&+ c_qQ_{pq}\overline{\Big(\frac{-c_q(\Im \lag \tilde z, y\rag+Q_{qq}C_{qq})}{c_pQ_{qp}}\Big)}\\=& c_pQ_{pp}\Big(d_0+\frac{2\Im( \lag y, e_{p}\rag \lag e_{p},\tilde z\rag)}{ Q_{pp}  
}\Big)\\&+ c_qQ_{pq} {\Big(\frac{-\overline c_q \Im \lag \tilde z,y\rag  }{\overline c_pQ_{qp}}\Big)}\\&+c_qQ_{pq} {\Big(\frac{-\overline c_q Q_{qq}C_{qq} }{\overline c_pQ_{qp}}\Big)}.
\end{align*}Now using (\ref{E:C22}) for  $m=q$, we rewrite this equality    in an equivalent form 
$$
(|c_p|^2 Q_{pp}-|c_q|^2 Q_{qq})d_0=A 
$$for some constant $A\in \R$. Thus if $\tilde z$ is such that $|c_p|^2 Q_{pp}-|c_q|^2 Q_{qq}\neq 0$, then we are able to find $C_{mk}$ satisfying (\ref{E:C3})-(\ref{E:C22}). 
If $|c_p|^2 Q_{pp}-|c_q|^2 Q_{qq}= 0$, then   linear system (\ref{E:hav2}), (\ref{E:ep2}) is not controllable, since for any $u\in \te$ and $t\ge 0$ we have
\begin{align*}
\frac{\dd}{\dd t}   &\Im \lag R_t( 0,u),  c_pe^{-i \la_pt}e_p-c_qe^{-i \la_qt}e _q\rag  \\=&   \Im \lag   i(\frac{\p^2}{\p x^2}-V) R_t(0,u) -iuQ(c_pe^{-i \la_pt}e_p+c_qe^{-i \la_qt}e _q), c_pe^{-i \la_pt}e_p-c_qe^{-i \la_qt}e _q\rag\\&+ \Im \lag      R_t(0,u) , i(\frac{\p^2}{\p x^2}-V)(c_pe^{-i \la_pt}e_p-c_qe^{-i \la_qt}e _q)\rag \\=&   \Im \lag   -iuQ(c_pe^{-i \la_pt}e_p+c_qe^{-i \la_qt}e _q), c_pe^{-i \la_pt}e_p-c_qe^{-i \la_qt}e _q\rag\\=&-u(|c_p|^2 Q_{pp}-|c_q|^2 Q_{qq})=0. \end{align*} 
 This non-controllability property is a remark of Beauchard and Coron \cite{BeCo}.

  \vspace{6pt}\textbf{Case 3.}  Here we suppose that    $\tilde z= \sum_{j=1}^{+\ty} c_je_j $ with $c_pc_qc_r\neq 0$, and $p,q,r$ are not equal to each other. If we define again $C_{mp}$, $m\neq p$ by (\ref{E:Cmp}) and $C_{mk}=0$ for any $k\ge1$ such that   $k\neq m,p$, then the arguments  of case 2 give the following equation for $d_0$
 $$
(|c_p|^2 Q_{pp}-\sum_{m\neq p}|c_m|^2 Q_{mm})d_0=\tilde A 
$$for some constant $\tilde A\in \R$. This implies that for any  $\tilde z$   such that 
$
 |c_p|^2 Q_{pp}-\sum_{m\neq p}|c_m|^2 Q_{mm}\neq 0,
$ we can find $C_{mk}$ satisfying (\ref{E:C3})-(\ref{E:C22}). 
Let us suppose that 
\begin{equation}\label{E:Lpq}
 |c_p|^2 Q_{pp}-\sum_{m\neq p}|c_m|^2 Q_{mm}=0.\end{equation}In this case, we define $C_{mp}$ by (\ref{E:Cmp}) only for integers  $m\ge1$ such that    $m\neq p,q,r$ and $C_{mk}=0$ for any $k\ge1$ such that   $k\neq m,p,q,r$. Then all the equations in (\ref{E:C21}) are verified, except for $m=p,q,r$.
We take any $ C_{qp}\in \C$ and choose $C_{qr}$ and $C_{rp}$   such that 
\begin{align}
 c_pQ_{rp}C_{rp}+ c_qQ_{rq}C_{rq}+c_pQ_{rr}C_{rr}&=-c_r \Im \lag \tilde z, y\rag,\label{E:A1}\\ c_pQ_{qp}C_{qp}+ c_qQ_{qq}C_{qq}+c_rQ_{qr}C_{qr}&=-c_q \Im \lag \tilde z, y\rag.\label{E:A2}
\end{align} 
Replacing the value of $C_{qr}$ from (\ref{E:A2}) into (\ref{E:A1}), then the value of $C_{pr}$ from  (\ref{E:A1}) into  (\ref{E:C21}) with $m=p$, and using (\ref{E:C3}), we get the following equation for $d_0$
 $$
(|c_p|^2 Q_{pp}+|c_q|^2 Q_{qq}-\sum_{m\neq p,q}|c_m|^2 Q_{mm})d_0= \tilde {\tilde A} 
$$for some constant $\tilde {\tilde A}\in \R$. Equality (\ref{E:Lpq}) implies that 
$$
 |c_p|^2 Q_{pp}+|c_q|^2 Q_{qq}-\sum_{m\neq p,q}|c_m|^2 Q_{mm} =0$$ if and only if $|c_q|^2 Q_{qq}=0$, which is not the case: $c_q\neq 0, Q_{qq}\neq 0$. Thus solution  $d_0\in \R$ exists, and the sequence $C_{mk}$ is constructed for any $\tilde z\notin \EE$.

\subsection{Multidimensional case}

In  this section, we suppose that $D$ is the rectangle  $(0,1)^d $, $d\ge 1$      and $V  (x_1,\ldots,x_d)$ $ =V_{1} (x_1)+\ldots+ V_{d} (x_d)  $, $V_k   \in C^\ty( [0,1],\R)$. 
This subsection is devoted to the  study of the linearization of   (\ref{E:hav1}), (\ref{E:ep1}) around the trajectory $\UU_t(\tilde z,0) $:
\begin{align}
i\dot z &= -\Delta z+V(x)z  + u(t) Q(x)\UU_t(\tilde  z ,0),\,\,\,\,
\label{E:hav2L}\\
z\arrowvert_{\partial D}&=0,\label{E:ep2L}\\
 z(0,x)&=z_0.\label{E:sp2L}
\end{align}    The proof of Theorem \ref{T:Lin} does not work in the multidimensional case for a general $\tilde z$. Indeed,  the well-known  asymptotic formula for  eigenvalues $\la_{k,V}~\sim~ C_dk^{\frac{2}{d}}$ implies that the frequencies $\omega_{mk}$ are dense in $\R$ for space dimension $d\ge 3$. Thus the moment problem $\check u (\om_{mk})=d_{mk}$ cannot be solved in the space $L^1(\R_+,\R)$ for a general  $d_{mk}\in\tilde\ell^2 $. The asymptotic formula for  eigenvalues  implies that the moment problem cannot be solved also in this case  $d=2$. Clearly, this does not imply the non-controllability of linearized system. Let us prove the controllability of (\ref{E:hav2L}), (\ref{E:ep2L}) for $\tilde z=e_{k,V}$. See our forthcoming publication
for the  case of a general $\tilde z$ and for an application to the nonlinear control problem.

For  $\tilde z=e_{k,V}$ the mapping $R_\ty(0,u)$ is given by 
$$\lag R_{\ty}(0,u), e_{m,V}\rag=-i    Q_{mk} 
\check u (\om_{mk})$$ (cf. \ref{E:gcza}).
\begin{lemma}\label{L:gc2}
The mapping  $R_{\ty}(0,\cdot)$ is linear continuous   from $\te$ to $
T_{e_{k,V}}\cap \VV$, where $\VV$ is defined by (\ref{E:ddsah}).
\end{lemma}
\begin{proof}  \vspace{6pt}\textbf{Step 1.} Let us admit that for any $m_j,k_j\ge 1$, $j=1,\ldots,d$ we have
\begin{align}\label{E:reza}
\Big|\frac{(m_1\cdot\ldots\cdot m_d)^3}{(k_1\cdot\ldots\cdot k_d)^3} \langle Qe_{k_1,\ldots,k_d,V },e_{m_1,\ldots,m_d,V }\rangle\Big|\le C.
\end{align}
Then (\ref{E:gcza}), (\ref{E:reza}) and the Schwarz inequality imply that
\begin{align*} 
&\|   R_{\ty}(0,u)\|_\VV^2=\sum_{m_1,\ldots,m_d=1}^{+\ty} |  {{m_1^3\cdot\ldots\cdot m_d^3}} \lag R_{\ty}(0,u), e_{{m_1,\ldots, m_d,V}}   \rag|^2\\&\le C\!\!\!\sum_{m=1}^{+\ty}\!\! \Big|\! {{m_1^3\cdot\ldots\cdot m_d^3}} \lag \tilde z, e_{{m_1,\ldots, m_d,V}}   \rag\langle Qe_{m,V },e_{m,V }\!\rangle \!\!\int^{+\ty}_0\!\!\!\!  u(s)\dd s\Big|^2\!\\&\quad+ C \!\| \tilde z\|_\VV ^2 \!\!\!\sum_{m,k=1, m\neq k }^{+\ty}\!\! \Big|\!\frac{(m_1\cdot\ldots\cdot m_d)^3}{(k_1\cdot\ldots\cdot k_d)^3}\! \langle Qe_{k_1,\ldots,k_d,V },e_{m_1,\ldots,m_d,V }\!\rangle \!\!\int^{+\ty}_0\!\!\!\! e^{i\om_{mk} s} u(s)\dd s\Big|^2\\&\le C\| \tilde z\|_\VV ^2 \|u\|_\te^2<+\ty.
\end{align*} 

 \vspace{6pt}\textbf{Step 2.} Let us prove (\ref{E:reza}). 
To simplify   notation, let us suppose that $d=2$; the proof of the general case is similar.   Let $V(x_1,x_2)=V_1(x_1)+V_2(x_2)$. Integration by parts gives
\begin{align*}
\lag Q e_{k_1,k_2,V} ,  e_{m_1,m_2,V} \rag=&\frac{1}{\la_{m_1,V_1}^2} \lag   (-\frac{\p^2}{\p x_1^2} +V_1) (Q e_{k_1,k_2,V}),  (-\frac{\p^2}{\p x_1^2} +V_1) (e_{m_1,m_2,V} ) \rag \nonumber\\=& \frac{1}{\la_{m_1,V_1}^2} (\int_0^1-2\frac{\p Q}{\p x_1}  \frac{\p e_{k_1,k_2,V}}{\p x_1}    \frac{\p e_{m_1,m_2,V} }{\p x_1} \Big|_{x_1=0}^{x_1=1}\dd x_2\nonumber\\&+ \lag   \frac{\p}{\p x_1}   (-\frac{\p^2}{\p x_1^2} +V_1) (Q e_{k_1,k_2,V}),   \frac{\p e_{m_1,m_2,V}}{\p x_1}  \rag\nonumber\\&+ \lag   (-\frac{\p^2}{\p x_1^2} +V_1) (Q e_{k_1,k_2,V}),   V_1e_{m_1,m_2,V}  \rag)\nonumber\\ =:&I_1+I_2+I_3.
\end{align*} 
Again integrating by parts, we get
\begin{align} 
I_1= &\frac{-2}{\la_{m_1,V_1}^2 \la_{m_2,V_2}^2} \!\!\int_0^1\!\!(-\frac{\p^2}{\p x_2^2} +V_2)  (\frac{\p Q}{\p x_1}  \frac{\p e_{k_1,k_2,V}}{\p x_1} )(-\frac{\p^2}{\p x_2^2} +V_2)    \frac{\p e_{m_1,m_2,V} }{\p x_1} \Big|_{x_1=0}^{x_1=1}\dd x_2\nonumber\\=&\frac{-2}{\la_{m_1,V_1}^2 \la_{m_2,V_2}^2} ( -2\frac{\p^2 Q}{\p x_1\p x_2}  \frac{\p^2 e_{k_1,k_2,V}}{\p x_1\p x_2}    \frac{\p^2 e_{m_1,m_2,V} }{\p x_1\p x_2} \Big|_{x_1=0}^{x_1=1} \Big|_{x_2=0}^{x_2=1}  \nonumber\\ &+\int_0^1 \frac{\p }{\p x_2} (-\frac{\p^2}{\p x_2^2} +V_2)  (\frac{\p Q}{\p x_1}  \frac{\p e_{k_1,k_2,V}}{\p x_1} )  \frac{\p^2 e_{m_1,m_2,V} }{\p x_1\p x_2} \Big|_{x_1=0}^{x_1=1}\dd x_2\nonumber\\  &+\int_0^1(-\frac{\p^2}{\p x_2^2} +V_2) (\frac{\p Q}{\p x_1}  \frac{\p e_{k_1,k_2,V}}{\p x_1}) V_2    \frac{\p e_{m_1,m_2,V} }{\p x_1} \Big|_{x_1=0}^{x_1=1}\dd x_2) .\nonumber
\end{align}
 In view of (\ref{E:j1})-(\ref{E:app3}), this implies that
$$
\Big|\frac{(m_1\cdot\ldots\cdot m_d)^3}{(k_1\cdot\ldots\cdot k_d)^3} I_1\Big|\le C.$$
The terms $I_2,I_3$ are treated   in the same way. We omit the details.

\end{proof}

 We rewrite (\ref{E:gcza}) in the form 
\begin{equation}\label{E:gczaL}
\check u (\om_{mk})= d_m,
\end{equation}where $d_m=\frac{\lag R_{\ty}(0,u), e_{m,V}\rag}{-i    Q_{mk} }$.   We have $\sum_{m=1, m\neq k}^{\ty} \frac{1}{|\om_{mk}|^d}<~+\ty$  for fixed $k\ge1$. Under Condition  \ref{C:p1}, (i), $d_m\in \ell^2_0$.
Applying Proposition \ref{P:P1}, we obtain the following theorem.
\begin{theorem}\label{T:LinL} Under Condition  \ref{C:p1},  the mapping
$R_{\ty}(0,\cdot):\te\rightarrow
T_{e_{k,V}}\cap \VV$   admits a continuous
right inverse, where the space $T_{e_{k,V}}\cap \VV$ is endowed with the
norm of $\VV$. \end{theorem}

   \subsection{Proof of Proposition \ref{P:P1}} 
 
  The construction of the operator $A$ is based on the following lemma.
  \begin{lemma}\label{L:Min} Under the conditions of Proposition \ref{P:P1},
  for any $d\in \ell^2_0 $ and $\e>0$, there is $u\in   {B_{ \te}(0,\e)}$ such that $\{\check u (\om_{m})\}=d$.
  \end{lemma}
  \begin{proof}[Proof of Proposition \ref{P:P1}] Let $d^n$ be any orthonormal basis   in $ \ell^2_0$. Applying Lemma \ref{L:Min}, we find a sequence $u_n\in     {B_{ \te}(0,\frac{1}{n})} $ such that  $\{\check u_n (\om_{m})\}= d^n$. For any $d\in  \ell^2_0 $, there is $c\in  \ell^2 $ such that $d=\sum_{n=1}^{+\ty}c_n d^n$. Let us define $A$ in the following way
  $$
  A(d)= \sum_{n=1}^{+\ty}c_n u_n.
  $$As $u_n\in   B_{ \te}(0,\frac{1}{n})$, this sum converges in $ \te $:
  $$
  \|A(d)\|_{ \te} \le \sum_{n=1}^{+\ty} |c_n|   \|u_n\|_{ \te}\le\Big(\sum_{n=1}^{+\ty} |c_n|^2   \Big)^\frac{1}{2}  \Big(\sum_{n=1}^{+\ty}    \|u_n\|_{ \te}^2 \Big)^\frac{1}{2}   \le C\| d\|_{  \ell^2_0}.
  $$ Thus $A:   \ell^2_0\ri  \te $ is  linear continuous and $\{ \check { A(d)} (\om_{m})\}=d$, by construction.
\end{proof} 
\begin{proof}[Proof of Lemma \ref{L:Min}] Let us take any $d\in \ell^2_0 $ and $\e>0$  and introduce the functional
$$
\HH(u):=\| \{\check u (\om_{m})\}-d\|_{ \ell^2_0} ^2=\sum_{m=1}^{+\ty}| \check u (\om_{m})  -d_{m}|^2
$$
defined on the space $ \te$.

\vspace{6pt}\textbf{Step 1.} First, let us show that  there is $u_0\in \overline {B_{ \te}(0,\e)} $ such  that
\begin{equation}\label{E:mh}
\HH(u_0)= \inf_{u\in \overline {B_{ \te}(0,\e)} } \HH(u).
\end{equation}To this end, let $u_n\in \overline {B_{ \te}(0,\e)}  $ be an arbitrary  minimizing sequence. Since $\BB \cap H^s(\R_+,\R)$ is~reflexive, without loss of generality, we can assume that  there is $u_0\in \overline {B_{\BB \cap H^s(\R_+,\R)}(0,\e)} $ such that $u_n \rightharpoonup u_0 $  in $\BB \cap H^s(\R_+,\R)$. Using the compactness of the injection $H^s([0,N])\ri C([0,N])$ for any $N>0$ and a diagonal extraction, we can assume that $u_n (t)\ri u_0(t)$ uniformly for $t\in[0,N]$. The Fatou  lemma implies that
$$
\int_0^{+\ty}|u_0(s)| \dd s\le \liminf_{n\ri\ty} \int_0^{+\ty} |u_n(s)|\dd s\le \e.
$$Again extracting a subsequence, if it is necessary, one gets $\{\check u_n(\om_{m})\} \rightharpoonup  \{\check u_0(\om_{m})\} $  in $  \ell ^2_0$ as $n\ri+\ty.$ Indeed, the tails on $[T,+\ty)$, $T\gg1$ of the integrals  (\ref{E:eaz}) are small uniformly in $n$ (this comes from the boundedness
of $u_n$ in $\BB$), and on the finite interval $[0, T]$ the convergence is uniform.)

This implies that $u_0\in  \te $ and 
$$
\HH(u_0)\le \inf_{u\in \overline {B_{ \te}(0,\e)} } \HH(u).
$$ The fact that  $u_0\in \overline {B_{ \te}(0,\e)} $ follows from the Fatou  lemma and lower weak semicontinuity  of norms. Thus we have (\ref{E:mh}).

\vspace{6pt}\textbf{Step 2.} To complete the proof, we need to show that $ \HH(u_0)=0.$  Suppose, by contradiction,  that $\HH(u_0)>0$.  As we shall see below, this implies that there is $v\in   \overline {B_{ \te}(0,\e)}  $ such that 
\begin{equation}\label{E:mdh}
\frac{\dd}{\dd t}\HH((1-t)u_0+tv)\Big|_{t=0}< 0.
\end{equation} 
Since
$(1-t)u_0+t v\in  \overline {B_{ \te}(0,\e)}$ for all $t\in[0,1]$, (\ref{E:mdh}) is a contradiction to (\ref{E:mh}).

To construct such a function $v$, notice that the derivative is given explicitly~by
\begin{align*}\frac{\dd}{\dd t}\HH((1-t)u_0+tv)\Big|_{t=0}&=  2 \sum_{m =1}^{+\ty} \Re[  (\check v(\om_{m}) -\check u_0(\om_{m}) ) \overline{( \check u_0(\om_{m})-d_{m} )} ]. \end{align*}
In view of this equality, the existence of $v$ follows immediately from the following lemma.

\begin{lemma}\label{L:Min2} Under the conditions of Proposition \ref{P:P1},
the set
$$  U:=\{  \{\check u(\om_{m})\}: u\in B_{ \te}(0,\e)\}
$$
is dense in $ \ell^2_0 $.
\end{lemma}
\begin{proof} Suppose that $h\in  \ell^2_0 $ is orthogonal to   $U$. Then  for any $u\in B_{ \te}(0,\e)\cap C_0^\ty((0,+\ty))$ we have    
\begin{align}\label{E:vvz} \sum_{m=1}^{+\ty} \check u(\om_{m})  \overline{h_{m}}=0.
\end{align}Replacing in this equality $\check u (\om_{m}) $ by its integral representation, we get   integrating by parts
\begin{align*}& 0=\sum_{m=1 }^{+\ty}  \int^{+\ty}_0 e^{i\om_{m}s} u(s)\dd s\overline{h_{m}}=\int^{+\ty}_0 P_p(s) u^{(p)}(s)\dd s\overline{h_{1}}\\&\quad+\sum_{m=2 }^{+\ty}  \int^{+\ty}_0 \frac{e^{i\om_{m}s} }{(-i \om_m)^p}u^{(p)}(s)\dd s\overline{h_{m}}\\ &=  \int^{+\ty}_0 u^{(p)}(s) \Big(   P_p(s)\overline{h_{1}}+\sum_{m =2 }^{+\ty} \frac{e^{i\om_{m}s} }{(-i \om_m)^p}\overline{h_{m}} \Big) \dd s =0,
\end{align*}where $P_p$ is a polynomial of degree $p\ge1$.
Since  this equality holds for any $u\in B_{\te}(0,\e)\cap C_0^\ty((0,+\ty))$, there is a polynomial   $\tilde P_{p-1}(s)$  of degree $p-1$ such that for any $s\ge0$ $$P_p(s)\overline{h_{1}}  +\sum_{m=2}^{+\ty}  \frac{e^{i\om_{m}s} }{(-i \om_m)^p}\overline{h_{m}} =\tilde P_{p-1}(s).$$ By Lemma \ref{L:Tza}, we have  $h_m=0$ for any $m\ge2$. Equality (\ref{E:vvz}) implies that $h_1=0$. This proves that $U$ is dense.
 \end{proof}
\end{proof}The following lemma is a generalization of 
Lemma 3.10 in \cite{VN}.
\begin{lemma}\label{L:Tza}
Suppose that $r_j\in\R^*$ and  $r_k\neq r_j$ for $k\neq j$ and  $P_p$ is a polynomial of degree $p\ge1$. If
\begin{equation}\label{E:mlml}
\sum_{j=1}^\ty c_je^{ir_js}=P_p(s)
\end{equation}
for any  $s\ge0$ and for some sequence  $c_j\in\C$ such that
$\sum_{j=1}^\ty |c_j|<\ty$,   then $c_j=0$ for all $j\ge1$ and $P_p\equiv 0$.
\end{lemma}
\begin{proof} Since the sum in the left hand side of  (\ref{E:mlml}) is bounded in $s$, the polynomial $P_p(s)$ is constant.
The case of constant right hand side follows from Lemma~3.10 in \cite{VN}.
 \end{proof}

\section{Controllability of nonlinear system}\label{S:NL}

\subsection{Well-posedness of Schr\"odinger equation}\label{S:APROX}

In this section, we suppose that $d=1$, $D=(0,1)$.  We   consider the following   Schr\"odinger equation
\begin{align}
i\dot z &= -\frac{\p^2z}{\p x^2} +V(x)z+ u(t) Q(x)z+v(t)Q(x)y,\,\,\,\, \label{E:hav5}\\
z\arrowvert_{\partial D}&=0,\label{E:ep5}\\
 z(0,x)&=z_0(x).\label{E:sp5}
\end{align}  
See Proposition 2 in \cite{KBCL} for the proof of      
well-posedness  of this system with $V=0$.  Here we prove well-posedness  in the case of $V\neq0$  and we give an estimate for the solution which is important for the study of the controllability property. 

 \begin{proposition}\label{L:lav}
 For any $z_0\in H^3_{(V)}$, $u,v\in L^1(\R_+,\R)\cap \BB $ and  $ y\in
C(\R_+,H^3_{(V)}) $,  
  problem
(\ref{E:hav5})-(\ref{E:sp5}) has a unique solution  $ z\in
C(\R_+, H^3_{(V)}) $. Furthermore,  there is a constant $C>0$ such that
\begin{align}
\sup_{t\in\R_+}\|z(t)\|_{3,V}&\le C( \| z_0\|_{3,V}+ \sup_{t\in \R_+} \|  y(t) \|_{3,V} (\|v\|_{L^1(\R_+)} +\|v\|_{\BB} ) )   \nonumber\\ & \quad\times \exp\Big(C(
\|u\|_{L^1(\R_+)}+1)\exp(\|u\|_{\BB}^2)\Big).  \label{E:S1}
\end{align}If $v=0$, then    for all $t\ge0$ we have
\begin{align}
\|z(t)\| =\|z_0\|.    \label{E:S2} 
\end{align}
 \end{proposition}
\begin{proof}
The proof follows the ideas of Proposition 2 in \cite{KBCL}. We give all the details for the sake of completeness.

Let us rewrite (\ref{E:hav5})-(\ref{E:sp5})  in the Duhamel form
\begin{equation}\label{E:DU}
z(t)=S(t) z_0-i\int_0^tS(t-s)[u(s)Qz(s)+v(s)Qy(s)]\dd s.
\end{equation} For any $u\in L^1(\R_+,\R)\cap\BB $ and     $z\in C(\R_+, H^3_{(V)})$, we estimate the function
$$
G_t(z):=\int_0^tS(-s)\big(u(s)Qz(s)\big)\dd s.
$$  
Integration by parts   gives (we write $\la_j, e_{j}$ instead of $\la_{j,V}, e_{j,V}$)
\begin{align*}
\lag Q z(s),   e_{j}    \rag=&\frac{1}{\la_{j}  } \lag (-\frac{\p^2}{\p x^2}+V)(Q z),  e_{j}  \rag      \nonumber\\=&\frac{1}{\la_{j}^2 } \lag (-\frac{\p^2}{\p x^2}+V)(Q z),  (-\frac{\p^2}{\p x^2}+V)e_{j }   \rag =\frac{1}{\la_{j }  ^2 }      \frac{\p^2}{\p x^2} (Q z)  \frac{\p}{\p x} e_{j}\big|_{x=0}^{x=1}\nonumber \\&+ \frac{1}{\la_{j}^2 }(\lag V(-\frac{\p^2}{\p x^2}+V)(Q z),  e_{j}   \rag+\lag \frac{\p}{\p x}(-\frac{\p^2}{\p x^2}+V)(Q z),    \frac{\p}{\p x} e_{j}    \rag) \nonumber \\ =&:I_{j}+J_j.
\end{align*}  
Thus
\begin{align} \label{E:gum}
\|G_t(z)\|_{3,V}^2 &=\sum_{j=1}^{+\ty}   \Big(j^3\int_0^t e^{i\la_{j}  s}u(s)  \lag Q z(s), e_{j}  \rag\dd s  \Big)^2 \nonumber\\&=\sum_{j=1}^{+\ty}   \Big(j^3\int_0^t e^{i\la_{j}  s}u(s)  (I_j+J_j)\dd s  \Big)^2.\end{align}Using (\ref{E:app3}), we get
$$
 \lag \frac{\p}{\p x}(-\frac{\p^2}{\p x^2}+V)Q z,    \frac{\p}{\p x} e_{j}   \rag= j\pi\lag \frac{\p}{\p x}(-\frac{\p^2}{\p x^2}+V)Q z,    \sqrt{2}\cos(j\pi x) \rag+s_j(z),$$ where $|s_j(z)|\le C \|z\|_{3,V}$ for all $j\ge1$. 
 The definition of $J_j$, the fact that $\{\sqrt{2}\cos(j\pi x)\} $ is an orthonormal system in $L^2$, (\ref{E:app1}) and the Minkowski inequality yield 
\begin{align}\label{E:zaz}
 \sum_{j=1}^{+\ty}   \Big(j^3\int_0^t e^{i\la_{j}   s}u(s)  J_j\dd s  \Big)^2\le C \Big( \int_0^t |u(s)|\|  z(s) \|_{3,V} \dd s \Big)^2.
\end{align}
     On the other hand, (\ref{E:app3}) implies that 
     $$
     \frac{\p^2}{\p x^2} (Q z)  \frac{\p}{\p x} e_{j} \big|_{x=0}^{x=1}=  j\pi\frac{\p^2}{\p x^2} (Q z)   \sqrt{2}\cos(j\pi x)
     \big|_{x=0}^{x=1} +\tilde s_j(z)=:jc_j(z)+ \tilde s_j(z),
     $$where $|\tilde s_j|\le C \|z\|_{3,V}$ for all $j\ge1$. Again applying the Minkowski inequality, we obtain  
     \begin{align}\label{E:zazz}
 \sum_{j=1}^{+\ty}   \Big(\frac{j^3}{\la_j^2}\int_0^t e^{i\la_{j} s}u(s) \tilde s_j(z)\dd s  \Big)^2\le C \Big( \int_0^t |u(s)|\|  z(s) \|_{3,V} \dd s \Big)^2.
\end{align}
     Since  $c_j(z)$ depends on the parity of $j$, without loss of generality, we can assume that $c(z):=c_j(z)$ does not depend on  $j$.
Thus  we cannot conclude as in the case of $J_{j}$. Here we use the fact that $u\in \BB$.  
Let $P\ge 1$ be an integer such that $P \le t< P+1 $. Using the Cauchy--Schwarz   and the Ingham inequalities, we obtain 
 \begin{align}    \sum_{j =1}^{+\ty}   & \Big( \int_0^{t}  e^{i  \la_{j}s }u(s)c(z)\dd s  \Big)^2 = \sum_{j=1}^{+\ty}  \Bigg(  \Big( \int_{P }^{t}+ \sum_{p =1}^{P}    \int_{ p-1  }^{p  } \Big)  e^{i  \la_{j}s }u(s)c(z)\dd s  \Bigg)^2 \nonumber
  \end{align}   \begin{align}&
 \le 2\sum_{j=1}^{+\ty}  \Bigg(   \int_{P }^{t}  e^{i  \la_{j}s }u(s)c(z)  \dd s  \Bigg)^2\nonumber\\& \quad+2\sum_{j =1}^{+\ty}   \Bigg(\sum_{p =1}^{P}    \frac{1}{ p^2} \Bigg) \Bigg(\sum_{p =1}^{P}     p^2\Big( \int_{ p-1  }^{p  }  e^{i  \la_{j}s }u(s)c(z)  \dd s  \Big)^2  \Bigg)
 \nonumber\\&\le C  \|   u(s)c(z)\|_{L^2([P ,t])}^2  
 + C\sum_{p =1}^{P}  p^2   \sum_{j =1}^{+\ty}    \Big( \int_{ p-1 }^{p  }  e^{i  \la_{j}s }u(s)c(z) \dd s  \Big)^2   \nonumber\\&\le C \| u(s)c(z)\|_{L^2([P ,t])}^2  
 + C\sum_{p =1}^{P}  p^2   \| u(s)c(z)\|_{L^2([ p-1 ,p ])}^2 \nonumber\\&\le   C \int_0^t w (s)\|z(s)\|_{3,V}^2  \dd s. \nonumber
  \end{align}   where $w(s)= |u(s)|^2\chi_{[P ,t]}(s)+ \sum_{p =1}^{P}  p^2 |u(s)|^2\chi_{[ p-1 ,p ]}(s)$. Notice that 
  \begin{equation} \label{E:Jd}
  \int_0^t w (s)   \dd s\le  \|u\|_{\BB}^2 \quad \quad \quad\text{for all $t\ge 0$.} 
  \end{equation} 
Combining  (\ref{E:gum})-(\ref{E:Jd}), we get
\begin{align} \label{E:guma}
\|G_t(z)\|_{3,V} &\le C\Big(   \int_0^t w (s)\|z(s)\|_{3,V}^2  \dd s\Big)^\frac{1}{2}+  C \int_0^t |u(s)|\|  z(s) \|_{3,V} \dd s. 
\end{align} The  quantity 
$$
\tilde G_t(f):=\int_0^tS(-s)\big(v(s)Qy(s)\big)\dd s 
$$is estimated in a similar way
\begin{align} \label{E:guma2}
\|\tilde G_t\|_{3,V} &\le C\Big(   \int_0^t \tilde w (s)\|y(s)\|_{3,V}^2  \dd s\Big)^\frac{1}{2}+  C \int_0^t |v(s)|\|  y(s) \|_{3,V} \dd s\nonumber\\ &\le C\sup_{s\in [0,T]}\|  y(s) \|_{3,V} (\|v\|_{L^1(\R_+)} +\|v\|_{\BB} ), 
\end{align} where $\tilde w(s)= |v(s)|^2\chi_{[P ,t]}(s)+ \sum_{p =1}^{P}  p^2 |v(s)|^2\chi_{[ p-1 ,p ]}(s)$.

Existence of a solution is obtained easily from (\ref{E:guma}) and (\ref{E:guma2}), by a fixed point theorem (cf. Proposition 2 in \cite{KBCL}). Uniqueness follows from  (\ref{E:S1}).

  Let us prove  (\ref{E:S1}). From (\ref{E:DU}) and (\ref{E:guma}) we have 
 \begin{align*}& \|z(t)\|_{3,V}^2 \le C(\| z_0\|_{3,V}^2+  \|\tilde G_t\|_{3,V}^2+ \|G_t\|_{3,V}^2)\nonumber\\&\le C\Bigg(\| z_0\|_{3,V}^2+ \| \tilde G_t\|_{3,V}^2+ \int_0^t w (s)\|z(s)\|_{3,V}^2  \dd s  +   \Big(   \int_0^t |u(s)|\|  z(s) \|_{3,V} \dd s\Big)^2\Bigg).
\end{align*}The Gronwall inequality implies 
\begin{align*} 
\|z(t)\|_{3,V}^2&\le C\Bigg( \| z_0\|_{3,V}^2+   \| \tilde G_t\|_{3,V}^2+   \Big(   \int_0^t |u(s)|\|  z(s) \|_{3,V} \dd s\Big)^2 \Bigg)\\& \quad \quad \quad \quad \quad \quad \quad \quad \quad \quad \quad \quad \quad \quad \quad \quad \quad\times\exp\big(C \int_0^t w (s)   \dd s \big).
\end{align*} Taking the square root of this inequality, using (\ref{E:Jd}) and the Gronwall inequality, we obtain
\begin{align*} 
\|z(t)\|_{3,V} &\le C( \| z_0\|_{3,V} +\| \tilde G_t\|_{3,V}) \\& \quad \quad       \times\exp\Big(C( \int_0^t w (s)   \dd s + \int_0^t |u(s)|  \dd s\exp( \int_0^t  w(s)   \dd s))\Big)\nonumber\\&\le C( \| z_0\|_{3,V} +\| \tilde G_t\|_{3,V})      \exp\Big(C(\|u\|_{L^1(\R_+)}+1)\exp(\|u\|_{\BB}^2)\Big).
\end{align*} In view of (\ref{E:guma2}), this completes the proof of the proposition.\end{proof}
\begin{remark}Let us notice that, one should not expect to have a well-posedness property in any Sobolev space $H^k$ with controls in $L^1$. Indeed, exact controllability property in $H^3$, proved by  Beauchard and Laurent \cite{KBCL} in the case $d=1$,  implies that the problem is not well posed in spaces $H^{3+\sigma}$ for any $\sigma>0$ (a point $z_1\in H^3\setminus H^{3+\sigma}$ would  not  be accessible  from a point $z_0\in H^{3+\sigma}$).   Schr\"odinger equation is well posed in higher Sobolev spaces, when control $u$ is more regular. 
\end{remark}
 
\begin{corollary}\label{C:Ka}
Denote by
$\UU_t(\cdot,\cdot): H^{3}_{(V)}\times L^1(\R_+,\R)\cap \BB \ri  H^{3}_{(V)}$ the resolving operator of
(\ref{E:hav1}), (\ref{E:ep1}). Then $\UU_t(\cdot,\cdot)$ is locally Lipschitz continuous, i.e., for any $\de>0$ 
  there is $C>0$ such that
\begin{align}\label{E:ppz}
\sup_{t\in\R_+}\|\UU_t(z_0,u)-\UU_t(z_0',u')\|_{3,V}\le C\|(z_0,u)- (z_0',u')\|_{H^{3}_{(V)}\times L^1(\R_+,\R)\cap \BB}
\end{align} for all $(z_0,u), (z_0',u')\in B_{H^{3}_{(V)}\times L^1(\R_+,\R)\cap \BB}(0,\de)$, where $ L^1(\R_+,\R)\cap \BB$ is endowed with the norm
$\|\cdot\|_{ L^1(\R_+,\R)\cap \BB}:=\|\cdot\|_{L^1}+\|\cdot\|_{\BB} $. 
\end{corollary}
\begin{proof}
Notice that $z(t):=\UU_t(z_0,u)-\UU_t(z_0',u')$ is a solution of problem \begin{align*}
i\dot z &= -\frac{\p^2 z}{\p^2 x}  + u(t) Q(x)z+(u(t)-u'(t))Q(x) \UU_t(z_0',u'),  \\
z\arrowvert_{\partial D}&=0, \\
 z(0,x)&=z_0(x)-z_0'(x).  
\end{align*}  Applying Proposition  \ref{L:lav}, we get (\ref{E:ppz}).  
\end{proof}
 
 \subsection{Exact controllability in infinite time
}\label{S:APROX}

For any control $u\in\te$, problem (\ref{E:hav5}), (\ref{E:ep5}) is well-posed
 in Sobolev space $H^3_{(V)}$. 
 Equality (\ref{E:S2}) implies that
it suffices to consider the controllability properties of
(\ref{E:hav5}), (\ref{E:ep5}) on the unit sphere $S$ in $L^2$. Let $\UU_\ty(z_0,u)$ be the $H^{3}_{(V)}$-weak
$\omega$-limit set of the trajectory corresponding to  
control
$u\in \te$ and   initial condition $z_0\in H^{3}_{(V)}$:
\begin{equation}\label{E:oml}
\UU_\ty(z_0,u):=\{z\in H^{3}_{(V)}: \UU_{t_n} (z_0,u)  \rightharpoonup z\,\,\,\text{in $H^{3}_{(V)}$}
  \,\,\text{for some}\,\,  t_n\rightarrow+\ty \}.
\end{equation}
By (\ref{E:S1}), $\UU_t (z_0,u)$   is bounded in $H^{3}_{(V)}$, thus
$\UU_\ty(z_0,u)$ is non-empty.
\begin{definition}
We say that   (\ref{E:hav5}),  (\ref{E:ep5})  is exactly controllable in
infinite time in   subset $H\subset S$, if for any $z_0,z_1\in
H$
there is a control $u\in \te$ such that $z_1\in \UU_\ty(z_0,u).$\end{definition}

 Below theorem is one of the main results of   this
paper.
\begin{theorem}\label{T:AnvL}
Under Condition \ref{C:p1}, for any $\tilde z\in S\cap H^{3}_{(V)} $
 there is  
$\de>0$ such that problem (\ref{E:hav5}), (\ref{E:ep5}) is
exactly controllable in infinite time in    $ S\cap
B_{H^{3}_{(V)}}(\tilde z,\de)$.
\end{theorem}See Section \ref{S:1} for the proof. 
\begin{remark} Let us emphasize that the novelty of Theorem \ref{T:AnvL} with respect to the previous result proved for (\ref{E:hav5}), (\ref{E:ep5})   in \cite{VaNe}  (see Theorem 3.1)  is that the controllability here is realized with controls which have small norms. 
\end{remark}
 Working in higher Sobolev spaces, one can prove similar exact controllability results with more regular controls. For example:
\begin{theorem}\label{T:Anve}
Under Condition \ref{C:p1}, for any $\tilde z\in S\cap H^{3+\sigma}_{(V)}, \sigma\in (0,2] $ 
 there is  
$\de>0$ such that problem (\ref{E:hav5}), (\ref{E:ep5}) is
exactly controllable in infinite time in    $ S\cap
B_{H^{3+\sigma}_{(V)}}(\tilde z,\de)$
 with controls   $u\in W^{1,1}(\R_+,\R)  \cap H^s(\R_+,\R) $ for any $s\ge 1$. \end{theorem} 
 These local exact controllability properties imply the following   global exact controllability result.

 \begin{theorem}\label{T:Anv}
Under Condition \ref{C:p1}, problem (\ref{E:hav5}), (\ref{E:ep5})
is
exactly controllable in infinite time     in  $ S\cap H^{3}_{(V)}   $ in the following sense: for any  $z_0\in  S\cap H^{3+\sigma}_{(V)}, \sigma\in (0,2] $ and $z_1\in  S\cap H^{3}_{(V)} 
$
there is a control    $u\in L^1(\R_+,\R)$ such that $z_1\in \UU_\ty(z_0,u)$. \end{theorem}
\begin{proof}  Let $\gamma:~ [0,1]\ri S\cap H^{3}_{(V)}  $ be any continuous function such that $\gamma (0)=z_0$, $\gamma (1)=z_1$ and  $\gamma (s)\in H^{3+\sigma}_{(V)} $ for any $s\in [0,1)$. Using the compactness of the curve $\gamma$ and Theorem \ref{T:Anve}, we prove that there is a control $v$ and time $T>0$ such that $\UU_T(z_0,v)\in  
B_{H^3_{(V)}}(  z_1,\de_{z_1})$, where $\de_{z_1}>0$ is the constant in Theorem~\ref{T:AnvL} corresponding to $ z_1 $. This completes the proof.
\end{proof}

\begin{remark}

We do not know if problem  (\ref{E:hav1})-(\ref{E:sp1}) is well posed in the space~$\VV$ for $d\ge2$ with $\te$-controls. Well-posedness in $\VV$ with $u\in \te$ would imply the controllability of the multidimensional problem. The nonlinear problem's solution is in $\VV$ for more regular controls. 
\end{remark}

\subsection{Proof of Theorem \ref{T:AnvL} }\label{S:1}

The proof is based on an inverse mapping theorem for  multivalued functions. We       apply   the inverse mapping theorem established by Nachi and Penot \cite{KNJP}, which suits well to the setting of Schr\"odinger equation.   For the reader's convenience,    we recall the statement of their result in Appendix (see Theorem~\ref{T:KP}).

Let us first slightly modify the definition  (\ref{E:oml}) of the set $\UU_\ty(z_0,u)$. Let $T_n\ri +\ty$ be the sequence defined in Section \ref{S:Gcayin}. Define 
\begin{equation}\label{E:oml2}
\UU_\ty(z_0,u):=\{z\in H^{3}_{(V)}: \UU_{T_{n_k}} (z_0,u)  \rightharpoonup z\,\,\,\text{in $H^{3}$}
  \,\,\text{for some}\,\,   n_k\rightarrow+\ty \}.
\end{equation}
Consider the    multivalued function
\begin{align}
\UU_\ty(\cdot,\cdot):S\cap H^{3}_{(V)} \times \te&\ri  2^{S\cap H^{3}_{(V)}} , \nonumber \\
(z_0,u)&\ri \UU_\ty(z_0,u). \nonumber\end{align} 
Since the result of Nachi and Penot   is stated in the case of Banach spaces, we cannot apply it directly to $\UU_\ty$.
Following Beauchard and Laurent \cite{KBCL}, we project the system   onto the tangent space $T_{\tilde z}$. We apply Theorem~\ref{T:KP} to the following  multivalued function
\begin{align}
\tilde\UU_\ty(\cdot,\cdot) :T_{\tilde z}\cap H^{3}_{(V)} \times  \te&\ri 2^{T_{\tilde z}\cap H^{3}_{(V)}} , \nonumber \\
(z_0,u)&\ri   P \UU_\ty(P ^{-1} z_0,u) , \nonumber
\end{align} where $P $ is the orthogonal projection in $L^2$ onto $T_{\tilde z}$, i.e., $Pz= z-\Re\lag z, \tilde z\rag \tilde z, z\in L^2$. Notice that   $P ^{-1}:B_{T_{\tilde z}}(0,\de)\ri S$ is well defined for sufficiently small $\de>~0$.
By the definition of $T_n$, we have $\lim_{n\ri +\ty}\UU_{T_n}(\tilde z,0)=\tilde z $. Hence (\ref{E:oml2}) implies that $\UU_\ty(\tilde z,0)= \tilde z $ and  $\tilde\UU_\ty(0,0)=\{0\}$.
 If we show that   $\tilde\UU_\ty $ is strictly differentiable at $(x_0,y_0)$ with $x_0=(0,0)\in T_{\tilde z}\cap H^{3}_{(V)} \times  \te$ and $y_0=0\in   T_{\tilde z}\cap H^{3}_{(V)}$ (see Definition~\ref{D:1}),  and the derivative admits a right inverse, then Theorem~\ref{T:AnvL} will be proved  as a consequence of   Theorem \ref{T:KP}.
 
\begin{proposition}\label{P:1}
The multifunction  
  $\tilde\UU_\ty $ is strictly differentiable at $(0, 0)\in T_{\tilde z}\cap H^{3}_{(V)} \times  \te $ in the sense of Definition \ref{D:1}. Moreover, the differential is the mapping 
  \begin{align}
  R_\ty(\cdot,\cdot) :T_{\tilde z}\cap H^{3}_{(V)} \times  \te&\ri  {T_{\tilde z}\cap H^{3}_{(V)}  } , \nonumber \\
(z_0,u)&\ri    R_\ty(  z_0,u), \nonumber
\end{align}
 where $ R_\ty$ is defined in Section \ref{S:Gcayin}. \end{proposition}
 \begin{proof}[Proof of Theorem \ref{T:AnvL}] 
 \vspace{6pt}\textbf{Case 1.} Let us suppose that   $\tilde z\in S \cap H^{3}_{(V)} \setminus \EE$. For any $(z_0,u)\in B_{ T_{\tilde z}\cap H^{3}_{(V)} \times  \te}(0,\de)$, the set $\tilde\UU_\ty(z_0,u)$ is closed and non-empty, if $\de >0$ is sufficiently small.
   The mapping $  R_\ty $ is invertible in view of Theorem \ref{T:Lin}. Thus Theorem \ref{T:KP} completes the proof.
    \begin{remark}\label{R:11} Let us point out that in case 1 the controls $u$ can be chosen such that $u(0)= \ldots=u^{(s-1)}(0)=0$.
    \end{remark}

 \vspace{6pt}\textbf{Case 2.} In the case   $\tilde z\in S \cap H^{3}_{(V)} \cap  \EE$, the linearized system (\ref{E:hav2}), (\ref{E:ep2}) is not controllable, and $R_\ty$ is not invertible. 
 Controllability in   finite time  near $\tilde z $    is obtained combining the results  of \cite{BeCo} and \cite{KBCL}: there is a constant $\de>0$ and a time $T>0$ such that for any $z_0, z_1\in S\cap  B_{H^3_{(V)}}(\tilde z,\de)$ there is a control $v\in L^2([0,T], \R)$ verifying $\UU_T(  z_0,v)=z_1$. Let us prove that the problem is exactly controllable in infinite time in $S\cap  B_{H^3_{(V)}}(\tilde z,\de)$. Take any $z_1\in S\cap  B_{H^3_{(V)}}(\tilde z,\de)$ and let us show that there is a control $u\in\te$ such that $z_1\in \UU_\ty(\tilde z,u)$. Let us suppose first that $z_1\notin \EE $. Then, by case 1, there is $\de_{z_1}>0$ such that  exact controllability in infinite time holds in  $S\cap  B_{H^3_{(V)}}(  z_1,\de_{z_1})$. By exact controllability property in finite time and by an approximation argument, one can find a control $u_1\in C^\ty_0((0,T), \R)$ such that $\UU_T(\tilde z, u_1)\in  B_{H^3_{(V)}}( z_1,\de_{z_1})$. Thus  the existence of $u_1$ follows from case~1 and Remark \ref{R:11}.
 
 Now let us suppose that $z_1\in \EE$. Since   $\EE\subset \cap_{k=1}^\ty H_{(V)}^k$, by \cite{BeCo} and \cite{KBCL}, there is a control $u_1 \in C^s([0,T], \R)$ such that $\UU_T(\tilde z, u_1) =z_1 $ and $u(0)=\ldots =u^{(s)}(0)=u(T)=\ldots =u^{(s)}(T)=0$. Extending $u_1$ by $0$ on $[T,+\ty)$, we obtain $z_1\in \UU_\ty(\tilde z, u_1)$.
 
   \end{proof}
 \begin{proof}[Proof of Proposition \ref{P:1}] 
   It suffices to show that    for
any $\e> 0$ there exists $ \de > 0$ for which
\begin{equation}\label{E:hpz}
e(\tilde\UU_\ty( z_0,u) -R_\ty( z_0,u), \tilde\UU_\ty( z_0',u')-R_\ty( z_0',u'))\le \e  \|(z_0,u)-(z_0',u')\| _{T_{\tilde z}\cap H^{3}_{(V)} \times  \te},
\end{equation}
whenever $(z_0,u), (z_0',u') \in B_{T_{\tilde z}\cap H^{3}_{(V)} \times  \te}((0,0),\de)$. Here 
$e(\cdot,\cdot)$ stands for the Hausdorff distance (see Appendix for the definition). It is clear from the definition of $e(\cdot,\cdot)$, that (\ref{E:hpz}) follows  from the following stronger estimate
\begin{align*}  \sup_{t\in\R_+ }   \| \UU_{t }(P^{-1}z_0 ,u)   & - R_{t }( z_0,u)   -  \UU_{t }(P^{-1}z_0',u')+R_{t }( z_0',u')\|_{T_{\tilde z}\cap H^{3}_{(V)}}
\nonumber\\& \le\e  \|(z_0,u)-(z_0',u')\| _{T_{\tilde z}\cap H^{3}_{(V)} \times  \te}.
\end{align*}
To prove this estimate, notice that the function
  \begin{align*}
y(t)&:=  \UU_{t }(P^{-1}z_0 ,u)    - R_{t }( z_0,u)   -  \UU_{t }(P^{-1}z_0',u')+R_{t }( z_0',u')\nonumber 
 \end{align*}
 is a solution of the problem  \begin{align*}
i\dot y &= -\frac{\dd^2y}{\dd x^2} +  (u-u')Q\big(\UU_t(P^{-1}z_0,u)-\UU_t(\tilde z,0) \big) \\&\quad  \quad \quad \quad \quad \quad \quad \quad     \quad \quad \quad \quad \quad \quad \quad +u'Q \big(\UU_t(P^{-1}z_0,u)-\UU_t(P^{-1}z_0',u) \big),\\
y\arrowvert_{\partial D}&=0, \\
 y(0,x)&=  P^{-1}z_0     -  z_0     -  P^{-1}z_0' +  z_0'  . 
\end{align*} We have\begin{equation} \label{E:haz}
\|  y(0)  \|_{3,V} \le \e \|    z_0-  z_0'   \|_{3,V} 
\end{equation} for  any $z_0,z_0' \in B_{T_{\tilde z}\cap H^{3}_{(V)}  }(0 ,\de)$ and for sufficiently small $\de>0$.
   Using (\ref{E:S1}) (we use the version of the inequality with $v_1f_1+v_2f_2$ instead of $vf$),   Corollary \ref{C:Ka} and (\ref{E:haz}), we get 
\begin{align*}
\sup_{t\in\R_+}\| y(t)\|_{3,V} &\le C\big(\|  y(0)  \|_{3,V}+  \sup_{t\in \R_+} \|  \UU_t(P^{-1}z_0,u)-\UU_t(\tilde z,0)\|_{3,V}  \|u-u'  \|_{ \te}    \\&\quad +   \sup_{t\in \R_+} \|  \UU_t(P^{-1}z_0,u) -\UU_t(P^{-1}z_0',u)\|_{3,V}   \|u' \|_{ \te}   \big)  \nonumber\\ &\le C \big(\|  y(0)  \|_{3,V}\!+\!
(\|z_0\|_{3,V}\!+\!\|u \|_{ \te})  \|u\!-\!u'  \|_{ \te}\!+\!  \|z_0-z_0'\|_{3,V} \|u'  \|_{ \te}\big)\nonumber\\&\le \e \|(z_0,u)-(z_0',u')\| _{T_{\tilde z}\cap H^{3}_{(V)} \times  \te} \end{align*}
 for sufficiently small $\de$. This proves the proposition.  \end{proof}

\section{Non-controllability result}\label{S:HIMTA}

\subsection{Main result}\label{S:PREM}

In this section, we study the problem of non-controllability of
Schr\"odinger system (\ref{E:hav1})-(\ref{E:sp1}), where $D\subset\R^d$ is a bounded domain with smooth
boundary, $V, Q \in C^\ty(\overline{D},\R)$ are  arbitrary given
functions.
The following lemma establishes the well-posedness of system
(\ref{E:hav1})-(\ref{E:sp1}) in the space $L^2$.
\begin{lemma}\label{L:LD}

For any   $z_0\in L^2$  and for any
  $u\in L^1_{loc}(\R_+,\R)$,
  problem
(\ref{E:hav1})-(\ref{E:sp1}) has a unique solution $z\in
C(\R_+, L^2)$. Furthermore, the resolving operator
$\UU_t(\cdot,u):L^2\rightarrow~L^2 $ taking $z_0$ to
    $z(t)$ satisfies the relation
\begin{align}
\|\UU_t(z_0,u)\|&=\|z_0\|,\,\,\,t\ge0.\nonumber
\end{align}
\end{lemma}
See  \cite{CW} for the proof. Let us define the set of
attainability of system (\ref{E:hav1}), (\ref{E:ep1}) from an
initial point $z_0\in S$:
\begin{equation}\label{E:Att}
\aA(z_0):=\{\UU_t(z_0,u): \text{ for all $u\in
W^{1,1}_{\text{loc}}(\R_+,\R)$ and $t\ge0$ }\} .
\end{equation}

The following theorem is the main result of this section.
\begin{theorem}\label{T:NC} For any constant 
$k\in(0,d)$, 
any initial condition $z_0\in S$ and any ball $B\subset H^k_{(V)}$, we have
$$
\aA^c(z_0)\cap B\cap S\neq \varnothing.
$$
\end{theorem}

Let us emphasize that this theorem does not exclude exact controllability in $H^k_{(V)}$ with controls form a larger space than $W^{1,1}_{\text{loc}}(\R_+,\R)$.

The proof of this theorem is an adaptation of ideas of Shirikyan
\cite{SHEU} to the case of   Schr\"odinger equation. Using a H\"older type  estimate for the solution of the equation, we show that the image by the resolving operator $\UU$ of a ball in the space of   controls  has a Kolmogorov $\e$-entropy strictly less than that  
of a ball $B$ in the phase space $H^k_{(V)}$. As we show, this implies the non-controllability.

\subsection{Some  $\e$-entropy estimates}

Let $X$ be a Banach space. For any compact set $K\subset X$ and
$\e>0$, we denote by $N_\e(K,X)$  the minimal number of sets of
diameters $\le2\e$ that are needed to cover $K$. The Kolmogorov
$\e$-entropy  of $K$ is defined as $H_\e(K,X) = \ln N_\e(K,X)$.

Let $Y$ be another Banach space and let $f : K \rightarrow Y$ be
a
H\"older continuous function:
\begin{equation}\label{E:Lip}
\|f(u_1)-f(u_2)\|_Y\le L\|u_1-u_2\|_X^\theta \end{equation}for
any
$u_1,u_2\in K$ and for some constants $L>0$ and $\theta\in
(0,1)$.
The following lemma follows immediately from the definition of
$\e$-entropy (cf. Lemma 2.1 in~\cite{SHEU}).
\begin{lemma}\label{L:KE} For any compact set $K \subset X$ and
any function $f : K \rightarrow Y$
satisfying inequality (\ref{E:Lip}), we have
$$
H_\e(f(K),Y)\le H_{ (\frac{\e}{L} )^{\frac{1}{\theta}}}(K,X)
\text{ for all $\e>0$}.
$$
\end{lemma}
We also need the following two lemmas.
\begin{lemma}\label{L:W11} For any $T>0$ and for any closed ball
$B \subset W^{1,1}([0,T],\R)$, there is a constant $C>0$ such
that
$$
H_\e(B,L^1([0,T],\R))\le \frac{C}{\e}\ln\frac{1}{\e}.
$$
\end{lemma}
This is Proposition 2.3 in \cite{SHEU}.
  \begin{lemma}\label{L:Hk} For any   $k>0$ and   any closed ball $B:= \overline{B_{H^k_{(V)}}(z_0,r)}$ such that $B_{H^k_{(V)}}(z_0,r) \cap S \neq \varnothing$   there is a  constant $C>0$ such that
\begin{align}\label{E:edm2}
H_\e(B\cap S,H^{k-1} )\ge C  \Big(\frac{1}{\e}\Big)^{d}.
\end{align}
\end{lemma}
\begin{proof} It is well known that  
\begin{equation}\label{E:edm}
C_1  \Big(\frac{1}{\e}\Big)^{d}
 \le H_\e(B,H^{k-1} )\le C _2 \Big(\frac{1}{\e}\Big)^{d}
\end{equation} for some constants $C_1, C_2>0 $ (e.g., see \cite{EDTR}). Consider the mapping
\begin{align}
f: [\frac{1}{2}, \frac{3}{2}]\times B\cap S&\ri H^{k-1},\nonumber\\
(s,z)&\ri sz.\nonumber
\end{align}  The set $f( [\frac{1}{2}, \frac{3}{2}]\times B\cap S)$ has a non-empty interior, so there is a ball $\tilde B $ in $H^k$ such that  
\begin{align}\label{E:edm3}
\tilde B\subset f( [\frac{1}{2}, \frac{3}{2}]\times B\cap S).
 \end{align}Clearly,
$$
\|f(s_1,z_1)-f(s_2,z_2)\|_{k-1}\le C(|s_1-s_2|+ \|  z_1 -  z_2\|_{k-1}).
$$ Using (\ref{E:edm3}) and Lemma \ref{L:KE}, we get  
\begin{align}
H_\e(\tilde B,H^{k-1})& \le H_\e( f(  [\frac{1}{2}, \frac{3}{2}]\times B\cap S),H^{k-1}) \nonumber\\
&\le H_{\frac{ \e}{C}}(    [\frac{1}{2}, \frac{3}{2}]\times B\cap S,\R\times H^{k-1})\nonumber\\
&\le H_{\frac{\e}{C}}(    [\frac{1}{2}, \frac{3}{2}] ,\R )+ H_{\frac{ \e}{C}}(    B\cap S, H^{k-1}) \nonumber\\
 &\le C  \big(\ln  \frac{1}{\e} + H_{\e} (    B\cap S, H^{k-1})\big). \nonumber
\end{align}Combining this with (\ref{E:edm}) for $\tilde B$, we obtain (\ref{E:edm2}).
 \end{proof}

\subsection{Proof of Theorem \ref{T:NC}}

Let us suppose, by contradiction, that there is $k\in(0,d)$, an
initial point $z_0\in S$   and a ball $B\subset H^k_{(V)}$ such
that
\begin{equation}\label{E:gd}
 B\cap S\subset \aA(z_0)   ,
\end{equation}where $\aA$ is the set of attainability of system
(\ref{E:hav1}), (\ref{E:ep1}) from the initial point $z_0$
defined
by (\ref{E:Att}). Let us set
\begin{align}
B_m:&=[0,m]\times B_{W^{1,1}([0,m],\R)}(0,m),\nonumber\\
\UU (B_m):&=\{\UU_t(z_0,u):\text{ for all }(t,u)\in
B_m\}.\nonumber
\end{align}We have
\begin{align}
 \R\times W^{1,1}_{loc}(\R_+,\R)&=\bigcup_{m=1}^\ty
 B_m,\nonumber\\
 \aA(z_0)&=\bigcup_{m=1}^\ty \UU ( B_m).\label{E:Ba}
\end{align}Combining (\ref{E:gd}), (\ref{E:Ba}) and the Baire
lemma, we see that there is a ball $Q\subset H^k_{(V)}$ and an
integer $m\ge1$ such that $\UU ( B_m)$ is dense in $Q\cap S$
  with respect to $H^k$-norm.

 \vspace{6pt}\textbf{Step 1.} Let us define the set
$$ \tilde{B}_m=\{(t,u)\in B_m: \text{ such that }
\UU_t(z_0,u)\in Q\}.$$
 Here we prove that $ \tilde{B}_m$ is compact in   $[0,m]\times
 L^1([0,m],\R).$ Indeed, take any sequence $(t_n,u_n)\in
\tilde{B}_m$. As $(t_n,u_n)\in {B}_m$ and $B_m$ is compact in
$[0,m]\times L^1([0,m],\R)$, there is a
 sequence $n_k\rightarrow\ty$ and $(t_0,u_0)\in B_m$ such that
 $$|t_{n_k}-t_0|+\|u_{n_k}-u_0\|_{L^1([0,m],\R)}\rightarrow0,
 k\rightarrow\ty.$$ We need to show that  $(t_0,u_0)\in
\tilde{B}_m$. As $\UU_{t_{n_k}}(z_0,u_{n_k})\in Q$, there is
$z\in
 Q$ such that $\UU_{t_{n_k}}(z_0,u_{n_k})\rightharpoonup z$ in
$H^k$ (again extracting a subsequence, if necessary). On the
other
 hand, Lemma \ref{L:LD} implies that
$\UU_{t_{n_k}}(z_0,u_{n_k})\rightarrow\UU_{t_0}(z_0,u_0)$ in
$L^2$. Thus $\UU_{t_0}(z_0,u_0)=z$ and  $(t_0,u_0)\in
\tilde{B}_m$. Thus $ \tilde{B}_m$ is compact in $[0,m]\times
 L^1([0,m],\R).$

In particular, this implies that $\UU (\tilde{B}_m)$ is compact
in
$L^2$,  as an image of a compact set by a continuous mapping. On the other hand, $\UU (\tilde{B}_m)$ is dense in the compact set
$Q\cap S$ in $L^2$. Thus $Q\cap S= \UU (\tilde{B}_m)$.

\vspace{6pt}\textbf{Step 2.} Using standard arguments, one can
show that we have
 $$
 \|\UU_t(z_0,u)-\UU_{t'}(z_0,u')\|\le
 C(|t-t'|+\|u-u'\|_{L^1([0,m],\R)}) 
 $$for any $(t,u), (t',u')\in
   \tilde{B}_m$, where $C>0$ is a constant not depending on $(t,u)$ and
$(t',u')$. Combining this with the interpolation inequality
$$\|z\|_{k-1}\le
 C\|z\|^{\frac{1}{k}}\|z\|_k^{\frac{k-1}{k}},$$we get
  $$
 \|\UU_t(z_0,u)-\UU_{t'}(z_0,u')\|_{k-1}\le
 C(|t-t'|^{\frac{1}{k}}+\|u-u'\|_{L^1([0,m],\R)}^{\frac{1}{k}})
 $$ for any $(t,u), (t',u')\in
\tilde{B}_m$. Here we used the fact that $\UU_t(z_0,u),
\UU_{t'}(z_0,u')\in
  Q$.
 Appying Lemmas \ref{L:KE} and \ref{L:W11} and of the
fact that $Q\cap S\subset\UU (\tilde{B}_m)$, we obtain 
\begin{align}
H_\e(Q\cap S, H^{k-1})&\le H_\e(\UU (\tilde{B}_m), H^{k-1}) \le
C H_{\e^k}(\tilde{B}_m, [0,m]\times L^1([0,m],\R))\nonumber\\&
  \le C  H_{\e^k}( {B}_m, [0,m]\times
L^1([0,m],\R))\nonumber\\&\le \frac{C}{{\e^k}}\ln
\frac{1}{{\e^k}}.\nonumber
\end{align}This estimate contradicts Lemma \ref{L:Hk} and proves the theorem.
 \begin{remark}
The same proof works also in the case of Schr\"odinger equation with any finite number of controls:
 $$
 i\dot z    = -\Delta z+V(x)z+ u_1(t) Q_1(x)z+\ldots+ u_n(t)
 Q_n(x)z,
 $$ where $n\ge1$ is any integer, $Q_j \in
C^\ty(\overline{D},\R)$ are arbitrary functions and $u_j$ are
the controls
 $j=1,\ldots,n$.
 \end{remark}

\section{Appendix}

\subsection{Genericity of Condition \ref{C:p1}}\label{S:GEN}

Let us assume that $D=(0,1)^d$ and introduce the space
\begin{align*} \GG:=\{V\in
C^\ty(D,\R):V(x_1,&\dots,x_d)=V_1(x_1)+\ldots+V_d(x_d)\nonumber\\&\text{for
some $V_k\in C^\ty([0,1],\R),k=1,\ldots,d$}\}. \end{align*} Then
$\GG$, endowed with the metric of   $  C^\ty(\overline{D},\R)$, is a closed subspace in   $  C^\ty(\overline{D},\R)$. By Lemma 3.12 in \cite{VN}, the set $\aA$ of all functions $V\in \GG$ such that property (ii) in Condition \ref{C:p1} is verified  is   $G_\delta$ set (i.e.,   countable intersection of dense open sets).  
 First let us prove genericity of property (i) in the case $d=1$.
 \begin{lemma}\label{L:aap}
 For any $V\in C^\ty([0,1],\R)$, the set of functions $Q\in   C^\ty([0,1],\R)$ such that  
 \begin{equation}\label{E:yu}
 \inf_{p , j \ge1}|{{ p^3 j  ^3}}    \langle Qe_{p ,V } ,e_{j ,V } \rangle |>0\end{equation} is dense  in $C^\ty([0,1],\R)$.      
 \end{lemma}
 \begin{proof} 
     If $V=0$, then 
  a straightforward calculation gives
\begin{equation*} \lag x^2 e_{p,0} , e_{j,0} \rag=\begin{cases} \frac{(-1)^{p+j}8pj}{\pi^2(p^2-j^2)^2} , & \text{if } p\neq j, \\ \frac{2}{3}-\frac{1}{ p^2\pi^2}, & \text{if } p=j,\end{cases}
\end{equation*} which implies (\ref{E:yu}) for $Q=x^2$ and $V=0$. In the general case, taking any $p\neq j$, we integrate by parts  (we write $\la_j, e_{j}$ and $z'', z'$ instead of $\la_{j,V}, e_{j,V}$ and $\frac{\dd^2z}{\dd x^2}, \frac{\dd z}{\dd x}$, respectively) 
\begin{align*}
 \lag Q e_{p} , e_{j}  \rag= & \frac{1}{\la_{j} } \lag  (-\frac{\dd ^2}{\dd x^2}+V)(Q e_{p}), e_{j} \rag\\=&\frac{1}{\la_{j}} (\lag -   Q '' e_{p}, e_{j}\rag +\lag  - Q'  e_{p}', e_{j}\rag+\la_{p}\lag   Q e_{p}, e_{j}\rag).
\end{align*}  This implies that
\begin{align}\label{E:hhsq}
 \lag Q e_{p}, e_{j}\rag&=   - \frac{1}{\la_{j}-\la_{p}} (\lag  Q''  e_{p} , e_{j}\rag +\lag    Q ' e_{p} ', e_{j}\rag) .
\end{align}
Again integrating by parts, we get
 \begin{align}\label{E:hhs}
 \lag    Q '     e_{p}', e_{j}\rag = &\frac{1}{\la_{j}} \lag  Q '     e_{p}', (-\frac{\dd ^2}{\dd x^2}+V) e_{j}\rag\nonumber\\ = &-\frac{1}{\la_{j}}  Q '     e_{p}'   e_{j}'\Big|_{x=0}^{x=1}+\frac{1}{\la_{j}} \lag (-\frac{\dd ^2}{\dd x^2}+V) (  Q '     e_{p}', e_{j}\rag  .
\end{align}
Notice that
 \begin{align*}   \lag (-\frac{\dd ^2}{\dd x^2}+V) ( Q '     e_{p}') , e_{j}\rag=&\lag  V     Q '     e_{p}' , e_{j}\rag + \lag  -     Q'''  e_{p}' , e_{j}\rag +\ \lag  -     Q''  e_{p}'' , e_{j}\rag\nonumber\\& +
  {\la_{p}}
 \lag   Q '     e_{p}', e_{j}\rag - \lag        Q'   (V   e_{p})', e_{j}\rag.
\end{align*}Replacing this into (\ref{E:hhs}), we get
\begin{align}\label{E:hhsa}
 \lag   Q '     e_{p}' , e_{j}\rag = &  \frac{1}{\la_{j}-\la_{p} } ( -  Q '     e_{p}'   e_{j}'\Big|_{x=0}^{x=1} + \lag  V     Q '     e_{p}' , e_{j}\rag+ \lag  -     Q'''  e_{p}' , e_{j}\rag  \nonumber\\&+ \lag  -     Q '' e_{p} '', e_{j}\rag    -\lag       Q '  (V   e_{p}  )', e_{j}\rag 
 ).
\end{align} Using (\ref{E:hhsq}) and  (\ref{E:hhsa}) and the fact that 
$$\lag  -   Q'' e_{p} '' , e_{j}\rag= -\lag  Q'' V e_{p}  , e_{j}\rag+\la_{p}\lag   Q ''  e_{p}, e_{j}\rag, $$ we obtain
\begin{align*}  \lag Q e_{p}, e_{j}\rag= &(- \frac{1}{\la_{j}-\la_{p}}  \lag  Q '' e_{p}, e_{j}\rag  -    \frac{\la_{p}}{(\la_{j}-\la_{p})^2}\lag   Q''   e_{p}, e_{j} \rag)  \nonumber\\&   -\frac{1}{(\la_{j}-\la_{p})^2}   ( -   Q '     e_{p}'  e_{j}'\Big|_{x=0}^{x=1}+ \lag  V          Q'       e_{p}', e_{j}\rag \nonumber\\& + \lag  -     Q'''  e_{p}' , e_{j}\rag  -\lag  Q'' V e_{p} , e_{j} \rag -\lag        Q'    (V   e_{p,} )', e_{j}\rag 
 ) \nonumber\\  
   =:&I_1+I_2.
\end{align*}Let $Q$ be such that $A:=     Q'(x)      \cos(p\pi x)     \cos(j\pi x)   \Big|_{x=0}^{x=1}\neq 0$. Clearly, this is verified for almost any $Q$, since $A$ depends only on the parity of $p$ and $j$. 
Let us choose  $Q$ such that $\lag Q e_{p}, e_{j}\rag \neq 0$ for all $p,j\ge1$; the set of such functions $Q$ is $G_\de$, by   Section 3.4 in \cite{VN}.
Using the estimates (\ref{E:app1})-(\ref{E:app3}), it is easy to see that $\inf_{p , j \ge1,p\neq j}|{{ p^3 j  ^3}}    I_2 |>0 $. Iterating the same arguments for $I_1$, we see that $\inf_{p , j \ge1,p\neq j}|{{ p^3 j  ^3}}   \langle Qe_{p ,V } ,e_{j ,V } \rangle   |>0 $ for almost any polynomial $Q$.

If $p=j$, using (\ref{E:app2}), we get
\begin{align*}  \lag Q e_{p} , e_{p}\rag= 2 \lag Q  ,  \sin^2  ( {p\pi} x  )\rag+ s_p,
 \end{align*}where $s_p\ri 0$. Thus
 \begin{align*}  \lag Q e_{p}, e_{p}\rag=  \lag Q  ,  1-\cos 2p\pi x\rag+ s_p=\int_0^1Q \dd x-\lag Q 
 , \cos 2p\pi x\rag+s_p.
 \end{align*}
 Taking $Q$ such that    $\int_0^1Q \dd x \neq 0$, we complete the proof of the lemma.
 
 \end{proof}
 Take any functions $Q_k\in C^\ty([0,1],\R)$, $k=1,\ldots, d$ in the dense set of Lemma \ref{L:aap} corresponding to $V_k\in C^\ty([0,1],\R)$, $k=1,\ldots, d$. Then $Q(x_1,\ldots, x_d)$ $:=Q_1(x_1)\cdot\ldots\cdot Q_d(x_d)$ satisfies (i) with $V(x_1,\ldots, x_d):=V_1(x_1)+\ldots+ V_d(x_d)$.

\subsection{Inverse mapping theorem for multifunctions}\label{S:IMT}

In this section, we recall the statement of the inverse mapping theorem for multivalued functions or multifunctions. We refer the reader to the paper \cite{KNJP} by Nachi and  Penot  for details and for a review of the literature on this subject.  

Let $X$ and $Y$ be Banach spaces. For any non-empty sets $C,D\subset X$, define the Hausdorff distance 
\begin{align}
d(x,D)&=\inf_{y\in D} \|x-y\|_X,\nonumber\\
e(C,D)&=\sup_{x\in C}d(x,D).\nonumber
\end{align}
We call a multifunction from $X$ to $Y$ any mapping $F$ from $X$ to $2^Y$.
\begin{definition}\label{D:1} A multifunction $F$   from an open set  $X_0\subset X$ to $Y$  is said to
be strictly differentiable at $(x_0, y_0)$ if there exists some continuous linear map $A : X \ri Y$ such that for
any $\e> 0$ there exist $\beta,\de > 0$ for which
\begin{equation} e(F(x)\cap B_Y(y_0,\beta)-A(x), F(x')-A(x'))\le \e  \|x-x'\|_X,\nonumber
\end{equation}
whenever $x,x'\in B(x_0,\de)$. The   map  $A$ is called a derivative of $F$ at $(x_0, y_0)$.
\end{definition} The following theorem is a generalization of the classical inverse function theorem to the case of multifunctions.
\begin{theorem}\label{T:KP}
Let $F$ be a multifunction  from an open set  $X_0\subset X$ to $Y$ with closed non-empty values. Suppose $F$ is strictly differentiable at $(x_0, y_0)\in  Gr(F)$, and
 some derivative $A$ of $ F$ at $(x_0, y_0)$ has a right inverse. Then   for
any neighborhood $U$ of  $ x_0$ there exists a neighborhood $ V$ of $y_0$ such that $V \subset F(U)$.
\end{theorem}
See Theorem 22 in \cite{KNJP} for the proof.

\end{document}